\documentclass[twoside]{article}
\newcommand{\letitre}{Asymptotic Hodge Theory of Vector Bundles}

\usepackage[margin=1.5cm,includeheadfoot]{geometry}

\usepackage{fancyhdr}
\fancypagestyle{plain}{%
   \fancyhead[LO]{}
   \fancyhead[CO]{}\fancyhead[RO]{}
}

\usepackage{amsmath}
\usepackage{amsthm}
\usepackage{amsfonts}
\usepackage{amssymb}
\usepackage{latexsym}

\usepackage[matrix,tips,graph,curve]{xy}

\makeatother

\makeatletter 
\@addtoreset{equation}{section}
\makeatother

\theoremstyle{plain}
\newtheorem{theorem}[equation]{Theorem}
\newtheorem{corollary}[equation]{Corollary}
\newtheorem{lemma}[equation]{Lemma}
\newtheorem{proposition}[equation]{Proposition}
\newtheorem{conjecture}[equation]{Conjecture}

\theoremstyle{definition}
\newtheorem{definition}[equation]{Definition}

\newtheorem{question}[equation]{Question}

\newcommand{\IC}{\mathbb{C}}

\newcommand{\IQ}{\mathbb{Q}}
\newcommand{\IR}{\mathbb{R}}

\newcommand{\End}{\mathrm{End}}
\newcommand{\tr}{\mathrm{tr}}
\newcommand{\Tr}{\mathrm{Tr}}
\newcommand{\ad}{\mathrm{ad}}

\newcommand{\Hom}{\mathrm{Hom}}

\newcommand{\ch}{\mathrm{ch}} % Chern Character
\newcommand{\Vect}{\mathrm{Vect}}
\newcommand{\even}{\mathrm{even}}
\newcommand{\odd}{\mathrm{odd}}

\newcommand{\rk}{\mathrm{rk}} 

\renewcommand\dim{{\rm dim\,}}

\def\d/{/\mspace{-6.0mu}/}

\newcommand{\db}{\bar{\partial}}

\newcommand{\p}{\partial}

\newcommand{\Prefix}{\mathcal{U}}

\newcommand{\LI}{L^{-1}}
\newcommand{\floor}[1]{\lfloor#1\rfloor}
\newcommand{\Vol}{\mathrm{Vol}}
\newcommand{\pp}[1]{\frac{\partial}{\partial #1}}

\usepackage[pdffitwindow=false,%
colorlinks=true,linkcolor=black,urlcolor=black,citecolor=black]{hyperref}

\begin{document}

\title{\letitre} %go to line 2 of file to change title
\author{Benoit Charbonneau and Mark Stern}

\date{November 2, 2011} 
\maketitle

\footnotetext{\parindent0mm BC: Department of Mathematics, St. Jerome's University in the University of Waterloo, 290 Westmount Rd N, Waterloo, Ontario, N2L 3G3, Canada; benoit@alum.mit.edu\\
MS: Department of Mathematics, Duke University, Durham, NC 27708, USA;  stern@math.duke.edu\\
BC is funded by NSERC Discovery grant.  MS is funded by NSF grant DMS 1005761.}

\section{Introduction}
 Let $M$ be a compact complex manifold of complex dimension $m$. Let $\Vect(M)$ denote the isomorphism classes of complex vector bundles over $M$, and let $\Vect(M,r)$ denote the subset of isomorphism classes of bundles of rank $r$. Given $E\in \Vect(M)$ equipped with a connection $A$ with curvature $F_A$, the \emph{Chern character} is defined to be 
\[\ch(E):= \left[\tr \exp(\frac{iF_A}{2\pi})\right]\in H^*(M,\IQ).\]
The Chern character extends to a surjective map \[\ch\colon \Vect(M)\otimes\IQ\twoheadrightarrow H^{\even}(M,\IQ).\]
When $M$ is K\"ahler, the Hodge decomposition, \[H^{d}(M,\IC)=\oplus_{p+q=d}H^{p,q}(M)\] and the Hodge filtration, 
\[S^p_HH^d:= \oplus_{j\geq p}H^{j,d-j}(M),\] are powerful tools in complex geometry. In this note, we introduce and analyze for projective varieties, $M$, a natural filtration $S_V^\cdot$ on $\Vect(M)$, which is analogous to the Hodge filtration on $H^{d}(M,\IC)$. Heuristically, the filtration measures degree of failure to admit a holomorphic structure. In fact, $S_V^0$ comprises bundles admitting holomorphic structures. 

Let $L$ be an ample holomorphic line bundle over a smooth projective variety $M$ of complex dimension $m$. Then $L$ admits a metric, $h$, whose induced Chern connection has curvature $F^L$ which satisfies $F^L(v,\bar v)>0$, for all nonzero holomorphic tangent vectors $v$. We call such a metric {\em admissible}. An admissible metric on an ample line bundle determines a K\"ahler form for $M$, by defining the K\"ahler form $\omega$ to be $\omega = iF^L$.  We call a choice of K\"ahler structure on $M$ induced by an ample line bundle with an admissible metric a \emph{polarization} of $M$.  We call $(L,h)$, (or $L$ when $h$ is understood) a \emph{polarizing line bundle}. 

Let $E$ be a complex vector bundle of rank $r$, with connection $A$. We do not assume that $E$ is holomorphic. The connection $A$ on $E$ and the Chern connection on a polarizing $L$ induce connections $A(k)$ on $E\otimes L^k$. Let $\db_{A(k)}$ denote the associated $\db$-operator. Define 
\[D_{A(k)} = \sqrt{2}(\db_{A(k)}+\db_{A(k)}^*).\]
For $k$ sufficiently large, the dimension of the kernel of $D_{A(k)}$ is the index of $D_{A(k)}^{\even}$, the restriction of $D_{A(k)}$ to $E\otimes L^k$ valued even forms. Let $s\in \ker (D_{A(k)})$. Let $s^j$ denotes the $(0,j)$ component of $s$, and write 
\[s = s^0 + s^2 +\cdots + s^{2\floor{\frac{m}{2}}},\]
where $\floor{x}$ denotes the integer part of $x$. 
When $E$ is holomorphic, the same estimates which imply that $\ker(D_{A(k)}^{\odd}) = 0$ for $k$ large, imply that $s=s^0$. Generically, for $E$ nonholomorphic, one does not expect to find any $s\in \ker (D_{A(k)})$ satisfying $s=s^0$. This leads us to the first definition of our filtration. 

%%%%%%% DEFINITION
\begin{definition} Let $(L,h)$ be a polarization of $M$. We say $E\in S^q_{V,L,h}\Vect(M)$ if $E$ admits a connection $ A $ so that for all $k$ sufficiently large, 
$s\in \ker (D_{A(k)})$ implies 
$s^{2j} = 0$, $\forall j>q$. We call $A$ an \emph{$S^q_{V,L,h}$-compatible} connection, or simply an $S^q_{V}$-compatible connection if we do not wish to specify the polarization data. We say $E\in S^q_{V}\Vect(M)$ if $E\in S^q_{V,L,h}\Vect(M)$ for {\em some} choice of polarization $(L,h)$. We say $E$ is of \emph{Hodge type} $q$ if $E\in S^q_V\Vect(M)\setminus S^{q-1}_V\Vect(M)$.  We say $E\in IS^q_{V}\Vect(M)$ if $E$ admits a connection $A$ which is $S^q_{V,L,h}$ compatible for {\em every} choice of polarization $(L,h)$. 
\end{definition}

For all $q$ it is easy to construct examples of bundles of Hodge type $q\leq \frac{m}{2}$ on complex $m$ manifolds. A bundle is of Hodge type $0$ if and only if it admits a holomorphic structure. For $0< q<\frac{m}{2}$, simply consider two projective varieties $M_1$ and $M_2$ equipped with bundles $E_1$ and $E_2$. Assume $E_1$ is holomorphic. Then the Hodge type of $E_1\times E_2$ is the Hodge type of $E_2$, by a separation of variables computation. If $\dim_{\IC}M_2=2q$, and $H^{0,2q}(M_2)\not = 0$ then there exists $E_2$ of type $q$ on $M_2$, thus yielding bundles of type $q$ on $M_1\times M_2$. In general, the conditions defining $S_V^\cdot$ appear to be very difficult to establish, and may be too rigid for many applications. Consequently, we introduce  quantized versions, $S_{V,p}^{\cdot}$ and $IS_{V,p}^{\cdot}$, of the filtrations, which are easier to treat in many applications. The quantized filtrations satisfy $S^{\cdot}_V\subset S_{V,j+1}^{\cdot}\subset S_{V,j}^{\cdot},$ $j= 1,2,\ldots$, and similarly for $IS_{V,\ast}^{\cdot}$. In order to motivate these new filtrations, we need two preliminary results. 

Let $\Pi$ denote the unitary projection onto $\ker(D_{A(k)}).$
  Let $P_j$ denote the projection onto $E\otimes L^k$ valued $(0,j)$-forms. 
Then (see \cite[Theorem 4.1.1]{MM}  or Proposition \ref{ma411} below) 
\begin{equation}\Tr\Pi  = \frac{k^m}{2^m\pi^m}\Vol(M)\rk(E) + O(k^{m-1}),\end{equation}
and (see Proposition \ref{trths} and Proposition \ref{naya})
\begin{equation}\label{hitr}\Tr P_{2j}\Pi  =   \frac{k^{m-2j}}{2^{m-2j}\pi^m(j!)^2}\|(F_A^{0,2})^{\wedge j}\|_{L_2}^2 + O(k^{m-2j-1}).\end{equation}

\begin{definition}We say $E\in S_{V,p,L,h}^q\Vect(M)$ if $E$ admits a connection $A$ satisfying $\Tr P_{2q+2}\Pi  = O( k^{m-2q-2-p})$, for all $k$ sufficiently large. 
 We call $A$ an \emph{$S_{V,p,L,h}^q$-compatible} connection, or simply $S_{V,p}^q$-compatible if we do not specify the polarization. We say $E\in S^q_{V,p} $ if $E\in S^q_{V,p,L,h} $ for {\em some} choice of polarization $(L,h)$. We say $E\in IS^q_{V,p} $ if $E$ admits a connection $A$ which is $S^q_{V,p,L,h} $ compatible for {\em every} choice of polarization $(L,h)$. We call such a connection $IS^q_{V,p}$ compatible. We say $E\in MS^q_{V,p} $ if $E\in S^q_{V,p,L,h}\Vect(M) $ for \emph{every} choice of polarization $(L,h)$ (but with $A$ possibly depending on the polarization).  
\end{definition}

For $q=0$, the filtrations all agree.  This is an immediate consequence of (\ref{hitr}), which shows that an $S^0_{V,1}$ compatible connection satisfies $F_A^{0,2} = 0$, and thus defines a holomorphic structure on $E$.  For $q>0$, (\ref{hitr}) also implies $S_{V,1}^q = IS_{V,1}^q$; in fact an $S_{V,1}^q$ compatible connection is also $ IS_{V,1}^q$ compatible. In general an $S_{V,p}^q$ compatible connection need not be $ IS_{V,p}^q$ compatible for $p>1$.

\begin{question}Do these filtrations eventually stabilize? In other words, is there some $N(q)$ so that for $p_1,p_2\geq N(q)$, $S_{V,p_1}^q = S_{V,p_2}^q$? 
\end{question} 
For $q=0$, the filtrations all stabilize at $N(0) = 1$ on holomorphic bundles. The stabilization question may therefore be thought of as an extended integrabilty condition. 
\medskip

In order to support our claim that $S_V^{\cdot}$ and its quantum extensions are analogous to the Hodge filtration on cohomology, we consider the compatibility of these filtrations under the Chern isomorphism. For line bundles, it follows from (\ref{hitr}) that the Chern character is compatible with the filtrations. 
\begin{theorem}\label{thm0}
\[\ch_p(S_V^q\Vect(M,1))\subset (S_H^{p-q}\cap \bar S_H^{p-q})H^{2p}(M,\IQ).\]
\end{theorem} 
We conjecture that this theorem extends to arbitrary rank.
\begin{conjecture}\label{conj1}\[\ch_p(S_V^q\Vect(M))\subset (S_H^{p-q}\cap \bar S_H^{p-q})H^{2p}(M,\IQ).\]\end{conjecture}
The conjecture is true for $IS_V^1$ for restricted $p$.  (See Proposition \ref{incpfA}.) 
\begin{theorem}\label{thmA} For $p<7$, 
\[\ch_p(IS_V^1\Vect(M))\subset (S_H^{p-1}\cap \bar S_H^{p-1})H^{2p}(M,\IQ).\]
\end{theorem}

For $q>1$, we have filtration compatibility for $p$ in a restricted range. (See Corollary \ref{corop1.5}.)
\begin{theorem}\label{thmB}
\[\ch_p(S_V^q\Vect(M))\subset (S_H^{p-q}\cap \bar S_H^{p-q})H^{2p}(M,\IQ),\,\,\,\forall p<q+3.\]
\end{theorem}

It is generally easier to prove results about the quantized versions of our filtration than for $S_V^{\cdot}$ directly. In fact, the preceding results follow from computations of the implications of inclusion in $S^{\cdot}_{V,j}$ and $IS^{\cdot}_{V,j}$, $j=1, 2$ and $3$. It is also easier to establish functorial properties of the quantized filtrations. For example, we have the following theorem, which is an immediate consequence of Equation \ref{hitr}. 
\begin{theorem}If $M$ and $N$ are two smooth projective varieties, and $f\colon M\rightarrow N$ is holomorphic, then $E\in S^q_{V,1}\Vect(N)$ implies $f^*E\in S^q_{V,1}\Vect(M).$
\end{theorem}

We have required $M$ to be projective in this discussion because the definitions of our filtrations required an ample holomorphic line bundle $L$. This is analogous to defining operations on cohomology only through the intermediary of harmonic forms. A metric free definition is essential for applications. Perhaps we should view the role of the polarization as providing 'enough' \emph{global} solutions to $D_{A(k)}s = 0$. More generally, we might define a filtration by requiring there to be 'enough' local solutions to $(\db_A+\db_A^*)s = 0$, with degree $s\leq 2q$, but this still requires a metric and seems unnatural. A completely metric free condition similar to $E\in IS^q_V$ is the following : there exists a local frame $\{s_a\}_a$ for $E$ so that $\db_A^{2q+1}s_a=0$ for all $a$. 

\begin{question}If $E\in IS_V^q \Vect(M)$, does there exist a local frame $\{s_a\}_a$ for $E$ so that $\db_A^{2q+1}s_a=0$ for all $a$? \end{question}
For $q=0$, the answer is, of course, yes. 

When $E$ is a holomorphic vector bundle (i.e. $E\in S_V^0\Vect(M)$), the Chern classes of $E$ are Poincar\'e dual to rational linear combinations of projective subvarieties.  Grothendieck's generalized Hodge conjecture suggests the following question.
\begin{question}If $E\in S_V^q \Vect(M,r)$, then is $c_p(E)$ Poincare dual to cycles supported in a finite union of codimension $p-q$ subvarieties?\end{question}
We have no evidence supporting a positive answer to this question.

The study of the large $k$-asymptotics of the Bergman kernel for $E\otimes L^k$  has been very fruitful when $E$ is holomorphic (for example, Donaldson \cite{DI}, Tian \cite{Ti2}, Wang \cite{W1}, Ma \cite{MM} and many others. See \cite{MM} for an extensive bibliography.) We were led to the structures examined here when investigating whether a similar analysis for nonholomorphic bundles might be used to improve our understanding of which bundles fail to admit a holomorphic structure. When we drop the assumption that $E$ is holomorphic, the natural analog of the Bergman kernel is the $L_2$ projection $\Pi$. See \cite{MM} for an extensive treatment of the Bergman kernel in both the holomorphic and nonholomorphic cases. In the holomorphic case, $\Pi$ is an endomorphism of sections of $E\otimes L^k$; in the nonholomorphic case, $\Pi$ is an endomorphism of $E\otimes L^k$ valued \emph{forms}. In particular, it defines maps $\Pi_0^q$ from sections to $E\otimes L^k$ valued $(0,2q)$ forms. For $q>0$, these maps lie deeper in the asymptotics of the Bergman kernel than have previously been computed, but are actually quite computable.  The above theorems all follow from computations of these asymptotics. It seems likely that  Conjecture \ref{conj1} can be proved to hold in a wider range by an extension of the computation of asymptotics which we have thus far undertaken. 

Unfortunately, the fact that the $q=0$ filtration stabilizes at $N(0) = 1$ limits the immediate application of this approach to discovering new obstructions to the existence of holomorphic structures. Nonetheless, we hope that these computations may be useful in other contexts.

%%%%%%%%%%%%%%%%
%%%%%%%%%%%%%%%%
\section{Approximating $\Pi$}\label{sec:approxPi}
 We retain the notation from the introduction. 
Fix a polarization $(L,h)$ for $M$. Let $\omega$ denote the corresponding K\"ahler form. We may then write 
\begin{equation}\label{admisscurv}F_{A(k)} = -ik\omega + F_A.\end{equation}
When we wish to emphasize the bundle rather than the connection, we will write $F^E$ for $F_A$. 

  Let $J$ denote the complex structure operator. Let $\{Z_j\}_{j=1}^m$ be a local frame for the holomorphic tangent bundle, and $\{w^j\}_{j=1}^m$ a dual coframe. We will be dealing with numerous curvature operators. 
Let $R$ denotes the curvature $2$ form induced by the Levi--Civita connection on the exterior algebra bundle. Set $F_{\bar j l} := F_{A(k)}(\bar Z_j, Z_l)+R(\bar Z_j, Z_l),$ and $\hat F_{\bar j l} := F_{A}(\bar Z_j, Z_l)+R(\bar Z_j, Z_l).$ 
 It is convenient to let $e(w)$ denote exterior multiplication on the left by the differential form $w$, and by $e^*(w)$, the adjoint operation. 
With this notation, a standard Lichnerowicz computation gives  
\begin{equation}\label{lich}D_{A(k)}^2 =  (\nabla^{0,1})^*\nabla^{0,1} - 2e(\bar w^{ j})e^*( \bar w^{ l})F_{\bar j l} + 2e(F_A^{0,2}) + 2e^*(F_A^{0,2}).\end{equation}

Expanding the second term on the right, we have on $(0,q)$ forms 
\begin{equation}\label{curvexp} - 2e(\bar w^{ j})e^*(\bar w^{ l})F_{\bar j l}
= 2kq  - 2e(\bar w^{ j})e^*(\bar w^{ l})\hat F_{\bar j l}.\end{equation}
 From (\ref{lich}) and (\ref{curvexp}), we see that 
on forms of odd degree, there exists $C_A>0$ independent of $k$ such that 
\begin{equation}\label{basest}\langle D_{A(k)}^2f,f\rangle_{L_2} \geq (2k-C_A)\|f\|^2.\end{equation}
The nonzero spectrum of $D_{A(k)}^2$ on even forms is the same as the nonzero spectrum of $D_{A(k)}^2$ on odd forms. 
Hence the spectrum of $D_{A(k)}^2$ on $E\otimes L^k$-valued even forms is contained in $\{0\}\cup [2k - C_A,\infty).$

The large spectral gap implies that for $k$ large and $Tk>>1$,  $e^{-TD_{A(k)}^2}$ is a good approximation to $\Pi$ in various operator norms, including the Trace, Hilbert--Schmidt, and supremum norms. For the convenience of the reader, we recall a few elementary features of the Trace and Hilbert--Schmidt norms, which we denote $\|\cdot\|_{Tr}$ and $\|\cdot\|_{HS}$ respectively.  For an operator $B$ with singular values $\lambda_j$,
\[\|B\|_{Tr} = \sum_{j}\lambda_j,\,\text{ and } \|B\|_{HS} = (\sum_{j}\lambda_j^2)^{\frac{1}{2}}.\]

If $B$ is given by integrating against a kernel $b(x,y)$, then 
\begin{equation}\label{hsn}\|B\|_{HS}^2 = \int \tr b^*(y,x)b(y,x)dydx.\end{equation} 
For bounded operators $A$ and $C$, 
\begin{equation}\label{trb}\|ABC\|_{Tr}\leq \|A\|_{sup}\|B\|_{Tr}\|C\|_{sup},\end{equation}
\begin{equation}\label{trhs}\|AB\|_{Tr}\leq \|A\|_{HS}\|B\|_{HS},\end{equation}
and
\begin{equation}\label{hsb}\|ABC\|_{HS}\leq \|A\|_{sup}\|B\|_{HS}\|C\|_{sup}.\end{equation}

The spectral gap and Equations (\ref{trb}) and (\ref{trhs}) imply for $k$ large and $Tk>>1$ that 
\begin{equation}\label{tr1}\|\Pi - e^{-TD_{A(k)}^2}\|_{Tr}\leq \frac{1}{k}\| D_{A(k)}^2e^{-TD_{A(k)}^2}\|_{Tr}\leq e^{-\frac{T}{2}k}\|e^{-\frac{T}{2}D_{A(k)}^2}\|_{Tr},\end{equation} 
and 
\begin{equation}\label{hs1}\|\Pi - e^{-TD_{A(k)}^2}\|_{HS}\leq \frac{1}{k}\| D_{A(k)}^2e^{-TD_{A(k)}^2}\|_{HS}\leq e^{-\frac{T}{2}k}\|e^{-\frac{T}{2}D_{A(k)}^2}\|_{HS}.\end{equation}
It follows from  Corollary \ref{hsq} and  Equation \ref{ptoqeq} that for $\frac{1}{4}\leq s\leq 1,$ $\exists c>0$ independent of $k$ large such that   $\|e^{-sk^{-1/2}D_{A(k)}^2}\|_{HS}\leq ck^{\frac{m}{2}}$, which immediately implies that $\|e^{-2sk^{-1/2}D_{A(k)}^2}\|_{Tr}\leq c^2k^m.$
Hence, the error in approximating $A\Pi C$ by $Ae^{-k^{-1/2}D_{A(k)}^2}C$, for $A$ and $C$ bounded
is $O(\|A\|_{sup}\|B\|_{sup}k^me^{-\frac{k^{1/2}}{2}})$ in the Hilbert--Schmidt, Trace, and supremum norms. 
 
 Equations (\ref{tr1}) and (\ref{hs1}) reduce  estimates of the errors introduced by replacing $\Pi$ by $e^{-k^{-1/2}D_{A(k)}^2}$ to estimates of heat kernels, and the problem of approximating $\Pi$ reduces to the familiar problem of approximating heat kernels.  Before embarking on the approximations, we first recall the standard estimates for the errors associated to heat kernel approximations. 

Let $k_t$ denote the heat kernel, meaning the Schwartz kernel for $e^{-tD_{A(k)}^2}$, and let $q_t$ denote an approximate kernel. Let $Q_t$ denote the operator corresponding to the kernel $q_t$. Set 
\[\epsilon_s := (\frac{\p}{\p s}+D_{A(k)}^2)q_s.\]
We write the difference of the  two kernels as 

\begin{equation}\label{duhamel}q_t-k_t = \int_0^tk_{t-s}\epsilon_sds.\end{equation}
Then 
\begin{equation}\label{duhamel2}\|Q_t-e^{-tD_{A(k)}^2}\|_{HS}\leq \int_0^t\|\epsilon_s\|_{HS}ds,\end{equation}
since $\|e^{-tD_{A(k)}^2}\|_{sup}\leq 1.$
In the next section we construct $q_t$ with $\epsilon_t$ small. 

%%%%%%%%%%%%%%%%%%%%%%%%%%%%%%%%
%%%%%%%%%%%%%%%%%%%%%%%%%%%%%%%%  SECTION
%%%%%%%%%%%%%%%%%%%%%%%%%%%%%%%%
\section{Approximate Heat Kernel}  
%%%%%  SUBSECTION
\subsection{Geodesic Coordinates} 
In this subsection and the next, for the convenience of the reader we gather elementary results on geodesic coordinates on Riemannian and K\"ahler manifolds. Let $x$ be geodesic normal coordinates centered at $y$, defined in some neighborhood $B_y$ of $y$. Let $r$ denote distance from $y$, and let $\frac{\p}{\p r}$ denote the radial vector field. In addition to the geodesic coordinate frame, it is useful to work with an orthonormal frame, 
$\{e_j\}_j$ satisfying $\nabla_{\frac{\p}{\p r}} e_j = 0,$ and at $y$, $e_j = \frac{\p}{\p x^j}$.  Define the operator 
\begin{equation}\label{definephi}\Phi(X):= \nabla_Xr\frac{\p}{\p r} - X.\end{equation}
It is easy to expand $\Phi$ recursively.  
\begin{align*}
r\nabla_{\frac{\p}{\p r}}\Phi(e_j) &= \nabla_{r\frac{\p}{\p r}} \nabla_{e_j}r \frac{\p}{\p r} \\
& = R (r \frac{\p}{\p r},e_j)r \frac{\p}{\p r} +   \nabla_{e_j}r \frac{\p}{\p r}  +   \nabla_{[r\frac{\p}{\p r},e_j]}r \frac{\p}{\p r} \\
& = R (r \frac{\p}{\p r},e_j)r \frac{\p}{\p r} +   e_j+\Phi(e_j) -   \nabla_{e_j+\Phi(e_j)}r \frac{\p}{\p r} \\
& = R (r \frac{\p}{\p r},e_j)r \frac{\p}{\p r}  -\Phi(e_j)-\Phi(\Phi(e_j)).	
\end{align*}
Hence, using the radially constant frame to define the integrals, we have 
\begin{equation}\Phi(e_j) = \frac{1}{r}\int_0^rs^2R(\frac{\p}{\p r},e_j)\frac{\p}{\p r}ds - \frac{1}{r}\int_0^r\Phi(\Phi(e_j))ds.\end{equation}
In particular, 
\begin{equation}\label{riemrad2}\Phi(X) =  \frac{r^2}{3}R(y)( \frac{\p}{\p r},X)\frac{\p}{\p r} 
+ \frac{r^3}{4}(\nabla_{\frac{\p}{\p r} }R)( \frac{\p}{\p r},X)\frac{\p}{\p r}+ O(r^4\nabla).
\end{equation}
Here we have Taylor expanded $R$ in a radially covariant constant frame. We write $R(y)$ for the radial parallel translation of $R$ from $y$. 

We may now compare our coordinate frame to the radially constant frame. (See for example \cite[Prop 1.28]{BGV}).
Compute 
\begin{align*}
r\frac{\p}{\p r}(e_jx^p) &= e_jx^p + [r\frac{\p}{\p r},e_j]x^p\\
 &= - \Phi(e_j)x^p \\
&= -\left[ \frac{r^2}{3}R(y)( \frac{\p}{\p r},e_j,\frac{\p}{\p r},e_m) 
+ \frac{r^3}{4}(\nabla_{\frac{\p}{\p r} }R)( \frac{\p}{\p r},e_j,\frac{\p}{\p r},e_m)\right]e_m(x^p)+ O(r^4).	
\end{align*} 
Hence
\[e_jx^p  = \delta_j^p +   \frac{1}{6}R(y)(r \frac{\p}{\p r},e_j,e_p,r\frac{\p}{\p r})  
+ \frac{1}{12}(\nabla_{r\frac{\p}{\p r} }R)(r \frac{\p}{\p r},e_j,e_p,r\frac{\p}{\p r}) + O(r^4),\]
and 
\begin{equation}\label{const2coord} e_j   = \frac{\p}{\p x^j}  +   \frac{1}{6}R(y)(r \frac{\p}{\p r},\frac{\p}{\p x^j},\frac{\p}{\p x^p},r\frac{\p}{\p r})\frac{\p}{\p x^p}  
+ \frac{1}{12}(\nabla_{r\frac{\p}{\p r} }R)(r \frac{\p}{\p r},\frac{\p}{\p x^j},\frac{\p}{\p x^p},r\frac{\p}{\p r})\frac{\p}{\p x^p} + O(r^4\nabla). \end{equation}
We write $O(r^4\nabla)$ above, instead of $O(r^4)$, to denote a vectorfield of magnitude $O(r^4).$
Inverting gives 
\begin{equation}\label{coord2const} \frac{\p}{\p x^j}   = e_j  -   \frac{1}{6}R(y)(r \frac{\p}{\p r},e_j,e_p,r\frac{\p}{\p r})e_p 
- \frac{1}{12}(\nabla_{r\frac{\p}{\p r} }R)(r \frac{\p}{\p r},e_j,e_p,r\frac{\p}{\p r})e_p + O(r^4\nabla). \end{equation}

The expansion for the metric in geodesic coordinates follows immediately from Equation (\ref{coord2const}) :
\begin{align}
\label{gij}g_{ij}(x) &= \delta_{ij} 
- \frac{1}{3}R(y)(\frac{\p}{\p x^i}, r\frac{\p}{\p r}, r\frac{\p}{\p r},\frac{\p}{\p x^j})
- \frac{1}{6}(\nabla_{r\frac{\p}{\p r} }R)(\frac{\p}{\p x^j},r \frac{\p}{\p r},r\frac{\p}{\p r},\frac{\p}{\p x^i})  +  O(r^4) ,\displaybreak[0]\\
\label{gammaij}\Gamma_{ij}^\mu(x) &= \frac{1}{3} R(y)(\frac{\p}{\p x^i}, r\frac{\p}{\p r},\frac{\p}{\p x^\mu} ,\frac{\p}{\p x^j})+   \frac{1}{3}R(y)(\frac{\p}{\p x^i}, \frac{\p}{\p x^\mu} ,r\frac{\p}{\p r},\frac{\p}{\p x^j})
  \\
&\quad \notag - \frac{1}{12} (\nabla_{e_j }R)(r \frac{\p}{\p r},\frac{\p}{\p x^\mu},\frac{\p}{\p x^i},r\frac{\p}{\p r}) - \frac{1}{12} (\nabla_{e_i }R)(r \frac{\p}{\p r},\frac{\p}{\p x^\mu},\frac{\p}{\p x^j},r\frac{\p}{\p r})+ \frac{1}{12} (\nabla_{\frac{\p}{\p x^\mu} }R)(r \frac{\p}{\p r},\frac{\p}{\p x^j},\frac{\p}{\p x^i},r\frac{\p}{\p r}) \\
&\quad \notag
 + \frac{1}{6} (\nabla_{r\frac{\p}{\p r} }R)(\frac{\p}{\p x^\mu},\frac{\p}{\p x^j},\frac{\p}{\p x^i},r\frac{\p}{\p r})
+  \frac{1}{6} (\nabla_{r\frac{\p}{\p r} }R)(r \frac{\p}{\p r},\frac{\p}{\p x^j},\frac{\p}{\p x^i},\frac{\p}{\p x^\mu})+O(r^3),\\
 \label{dvol}dv(x) &= \Bigl(1-\frac{1}{6}Ric(y)(r\frac{\p}{\p r},r\frac{\p}{\p r})+ O(r^3)\Bigr)dx^1\wedge\cdots\wedge dx^{2m}.
\end{align}

From these expansions, we also compute 
\begin{align*}
\Delta r^2 &= d^*d\sum_j(x^j-y^j)^2   = -2\sum_{i,j}e^*(dx^i)\nabla_{\frac{\p}{\p x^i}}((x^j-y^j)dx^j)\\
& = -2\sum_{j}e^*(dx^j) dx^j -2\sum_{i,j}(x^j-y^j)e^*(dx^i)\nabla_{\frac{\p}{\p x^i}}(dx^j)\\
&= -2\sum_{j}g^{jj}+2\sum_{i,j}(x^j-y^j)g^{il}\Gamma_{il}^j\\
&= -4m + \frac{2}{3}Ric(r\frac{\p}{\p r},r\frac{\p}{\p r}) +O(r^3). 
\end{align*}
In particular,
\begin{equation}\label{dr2}
4m+\Delta r^2 =  O(r^2).
\end{equation}

\subsection{Complex Structure in Normal Coordinates}
Let $J$ denote the complex structure operator.   
It is also convenient to define 
$J_0 \in C^{\infty}(B_y, \End(TM))$,  with $J_0^2 = -1$, as follows. Choose the geodesic coordinates to satisfy at $y$,
$J\frac{\p}{\p x^j} = \frac{\p}{\p x^{j+m}}$, for $1\leq j\leq m$. Define $J_0$ to extend this relation to all of $B_y$: 
\[J_0\frac{\p}{\p x^j} = \frac{\p}{\p x^{j+m}},\quad\text{ for } 1\leq j\leq m.\]
We introduce $J_0$-complex coordinates, $z^j = x^j-y^j+ i(x^{j+m}-y^{j+m}),$ $ 1\leq j\leq m$. Then 
\[r\frac{\p}{\p r} = z^l\frac{\p}{\p z^l} + \bar z^l\frac{\p}{\p \bar z^l}.\] 
Let $\{f_a:=\frac{1}{2}(e_a-ie_{a+m})\}_{a=1}^m$ be a radially covariant constant frame of the holomorphic tangent bundle, with dual coframe $\{\eta^a\}_{a=1}^m$.
Let $f_{\bar a} = \bar f_a,$ and $\eta^{\bar a}= \bar\eta^a.$ Then 
from Equation (\ref{const2coord}) we have 
\begin{equation}\label{cxconst2coord}\begin{split} f_a  & = \frac{\p}{\p z^a}  +  \frac{1}{3}R(y)(r \frac{\p}{\p r},\frac{\p}{\p z^a},\frac{\p}{\p z^p},r\frac{\p}{\p r})\frac{\p}{\p \bar z^p}  +  \frac{1}{3}R(y)(r \frac{\p}{\p r},\frac{\p}{\p z^a},\frac{\p}{\p \bar z^p},r\frac{\p}{\p r})\frac{\p}{\p z^p}  \\
&\quad + \frac{1}{6}(\nabla_{r\frac{\p}{\p r} }R)(r \frac{\p}{\p r},\frac{\p}{\p z^a},\frac{\p}{\p z^p},r\frac{\p}{\p r})\frac{\p}{\p \bar z^p}
+ \frac{1}{6}(\nabla_{r\frac{\p}{\p r} }R)(r \frac{\p}{\p r},\frac{\p}{\p z^a},\frac{\p}{\p \bar z^p},r\frac{\p}{\p r})\frac{\p}{\p z^p} + O(r^4\nabla), \end{split}\end{equation}
and 
\begin{equation}\label{cxcoord2const}
\begin{split} \frac{\p}{\p z^a} &= f_a  -  \frac{1}{3}R(y)(r \frac{\p}{\p r},\frac{\p}{\p z^a},\frac{\p}{\p z^p},r\frac{\p}{\p r})f_{\bar p}  -  \frac{1}{3}R(y)(r \frac{\p}{\p r},\frac{\p}{\p z^a},\frac{\p}{\p \bar z^p},r\frac{\p}{\p r})f_p  \\
&\quad
- \frac{1}{6}(\nabla_{r\frac{\p}{\p r} }R)(r \frac{\p}{\p r},\frac{\p}{\p z^a},\frac{\p}{\p z^p},r\frac{\p}{\p r})f_{\bar p}
- \frac{1}{6}(\nabla_{r\frac{\p}{\p r} }R)(r \frac{\p}{\p r},\frac{\p}{\p z^a},\frac{\p}{\p \bar z^p},r\frac{\p}{\p r})f_p + O(r^4\nabla). \end{split}\end{equation}

It is convenient to express $\db$ in a mixture of coordinate and radially covariant constant frames. We have 
\begin{equation}\label{db2d}
\begin{split}
\db = \bar \eta^a\Bigl[&
\nabla_{\frac{\p}{\p \bar z^a}}   
+  \frac{1}{3}R(y)(r \frac{\p}{\p r},\frac{\p}{\p \bar z^a},\frac{\p}{\p  z^p},r\frac{\p}{\p r})\nabla_{\frac{\p}{\p \bar z^p} } +  \frac{1}{3}R(y)(r \frac{\p}{\p r},\frac{\p}{\p \bar z^a},\frac{\p}{\p \bar z^p},r\frac{\p}{\p r})\nabla_{\frac{\p}{\p  z^p}} 
\\
&+ \frac{1}{6}(\nabla_{r\frac{\p}{\p r} }R)(r \frac{\p}{\p r},\frac{\p}{\p \bar z^a},\frac{\p}{\p z^p},r\frac{\p}{\p r})\nabla_{\frac{\p}{\p \bar z^p}}
+ \frac{1}{6}(\nabla_{r\frac{\p}{\p r} }R)(r \frac{\p}{\p r},\frac{\p}{\p \bar z^a},\frac{\p}{\p\bar z^p},r\frac{\p}{\p r})\nabla_{\frac{\p}{\p  z^p}}
\Bigr] + O(r^4\nabla).
\end{split}
\end{equation}

From (\ref{cxcoord2const}) and its conjugate, we can compute the difference between $J$ and $J_0$. We find
\[ (J-J_0)\frac{\p}{\p z^a}=   \frac{2i}{3}R(y)(r \frac{\p}{\p r},\frac{\p}{\p z^a},\frac{\p}{\p z^p},r\frac{\p}{\p r})\frac{\p}{\p \bar z^p} + \frac{i}{3}(\nabla_{r\frac{\p}{\p r} }R)(r \frac{\p}{\p r},\frac{\p}{\p z^a},\frac{\p}{\p z^p},r\frac{\p}{\p r})\frac{\p}{\p \bar z^p} + O(r^4\nabla). \]

Hence 
\begin{equation}\label{jj0}J = J_0 + \frac{2i}{3} R(y)(z^b\frac{\p}{\p z^b},\frac{\p}{\p \bar z^\mu},z^c\frac{\p}{\p z^c},\cdot) \frac{\p}{\p z^\mu}
- \frac{2i}{3} R(y)(\bar z^b\frac{\p}{\p \bar z^b},\frac{\p}{\p z^\mu},\bar z^c\frac{\p}{\p \bar z^c}, \cdot) \frac{\p}{\p \bar z^\mu}
 + \delta J,
\end{equation}
with
\begin{equation}\label{deltaJ}\delta J =  \frac{i}{3} (\nabla_{r\frac{\p}{\p r} }R)(z^b\frac{\p}{\p z^b}, \frac{\p}{\p \bar z^\mu}, z^c\frac{\p}{\p z^c}, \cdot) \frac{\p}{\p z^\mu}
- \frac{i}{3} (\nabla_{r\frac{\p}{\p r} }R)(\bar z^b\frac{\p}{\p \bar z^b}, \frac{\p}{\p z^\mu}, \bar z^c\frac{\p}{\p \bar z^c},\cdot) \frac{\p}{\p \bar z^\mu}+O(r^4). 
\end{equation}

As a useful special case, we note that 
\begin{multline}\label{jj0rdr}
%\begin{split}
(J-J_0)r\frac{\p}{\p r} = \frac{2i}{3} R(y)(z^b\frac{\p}{\p z^b},\frac{\p}{\p \bar z^\mu},z^l\frac{\p}{\p z^l},\bar z^c\frac{\p}{\p \bar z^c}) \frac{\p}{\p z^\mu}
- \frac{2i}{3} R(y)(\bar z^b\frac{\p}{\p \bar z^b},\frac{\p}{\p z^\mu},\bar z^l\frac{\p}{\p \bar z^l},z^c\frac{\p}{\p  z^c} ) \frac{\p}{\p \bar z^\mu}\\
 +\frac{i}{3} (\nabla_{r\frac{\p}{\p r} }R)(z^b\frac{\p}{\p z^b}, \frac{\p}{\p \bar z^\mu}, z^l\frac{\p}{\p z^l},\bar z^c\frac{\p}{\p \bar z^c}) \frac{\p}{\p z^\mu}
- \frac{i}{3} (\nabla_{r\frac{\p}{\p r} }R)(\bar z^b\frac{\p}{\p \bar z^b}, \frac{\p}{\p z^\mu}, \bar z^l\frac{\p}{\p \bar z^l},z^c\frac{\p}{\p  z^c} ) \frac{\p}{\p \bar z^\mu}+O(r^5\nabla). 
%\end{split}
\end{multline} 
 We summarize these computations with the following lemma relating $J$ and $J_0$. 
\begin{lemma}\label{JJ0}There exists $A_{bl}\in C^{\infty}(B_y, \Hom(T^{0,1},T^{1,0}))$ and $B_{\bar b\bar l}\in C^{\infty}(B_y,\Hom(T^{1,0},T^{0,1}))$ for which
\[J = J_0 + z^bz^lA_{bl} + \bar z^b\bar z^lB_{\bar b\bar l} + O(r^3).\]
\end{lemma}

We will also need the following elementary lemma. 
%%%%%%%%%%%% LEMMA
\begin{lemma}\label{Jric}
For all vector fields $X$,
\begin{align*}
Ric(X, JX) &= 0, \text{ and}\\
Ric(r\frac{\p}{\p r}, r\frac{\p}{\p r}) &= 2z^l\bar z^jRic(\frac{\p}{\p z^l},\frac{\p}{\p \bar z^j})+O(r^4).	
\end{align*}
\end{lemma}
\begin{proof}
Write $X = X^{1,0} + X^{0,1}$, with $JX^{1,0} = iX^{1,0}$ and $JX^{0,1} = -iX^{0,1}.$ Then 
since $Ric$ is symmetric,
\[Ric(X, JX) = iRic(X^{1,0} + X^{0,1}, X^{1,0} - X^{0,1})=0.\] 

For the second claim, we write   
\[Ric(z^a\frac{\p}{\p z^a}, z^b\frac{\p}{\p z^b}) = -iRic(z^a\frac{\p}{\p z^a}, J_0z^b\frac{\p}{\p z^b})=
-iRic(z^a\frac{\p}{\p z^a}, Jz^b\frac{\p}{\p z^b})+O(r^4)= O(r^4).\]
Similarly $Ric(\bar z^a\frac{\p}{\p \bar z^a}, \bar z^b\frac{\p}{\p \bar z^b}) =  O(r^4),$
and the claim follows upon expanding $r\frac{\p}{\p r} = z^a\frac{\p}{\p z^a} + \bar z^a\frac{\p}{\p \bar z^a}.$
\end{proof}

%%%%%%%%%%%% SUBSECTION
\subsection{Parallel Transport}
Before we construct an approximate heat kernel, it is useful to make a few observations about parallel translation. Let $S:= \Lambda^{0,\even}\otimes E\otimes L^k$, and $\pi_j\colon M\times M\rightarrow M$, $j=1$ or $2$, denote the projection onto the first or second factor.  Let the $j$-lift of $X$ be the unique element of $(d\pi_j)^{-1}(X)\cap \ker(d\pi_{j\pm 1})$. When no confusion will result, we will use the same symbol to denote a vector and its lift. A Schwartz kernel, $q_t$, for an approximate heat kernel is a section of the bundle \[V:= \Hom(\pi_2^*S,\pi_1^*S).\] 
 Thus $q_t(x,y)\in \Hom(S_y,S_x).$  In describing such kernels it is useful to identify $S_x$ and $S_y$ via parallel translation along distance minimizing geodesics. The kernel we construct will be supported in a small neighborhood of the diagonal in $M\times M$, making such an identification unique. Let $\gamma_{x,y}\colon [0,d(x,y)]\rightarrow M$ be the minimal unit speed geodesic from $y$ to $x$. 
Let $\psi(x,y)$ denote parallel translation  of $\psi\colon S_y\rightarrow S_x$ along $\gamma_{x,y}$. Let  $\frac{\p}{\p r_1}$ and $\frac{\p}{\p r_2}$  in $T_{x,y}(M\times M)$ denote the canonical $1$- and $2$- lifts to $M\times M$ of $\gamma_{x,y}'(d(x,y))$ and $-\gamma_{x,y}'(0)$ respectively. By definition, 
\begin{itemize}
\item $\psi(x,x) = \text{identity},$ and
\item $\nabla_{\frac{\p}{\p r_j }}\psi =0,\,\,\,j=1,2.$
\end{itemize}

It is convenient to factor $\psi$ as $\psi = \psi_{LC}\otimes \psi_E\otimes\psi_L=:\hat\psi\otimes \psi_L$, where $\psi_{LC}$, $\psi_{E}$, and $\psi_{L}$ denote parallel translation of sections of $\Lambda^{0,\even}$, $E$, and $L$ respectively. 
 The local geometry is largely encoded in $\psi$; hence we record some of its properties before constructing approximate kernels. 
 
Let $(x^i)$ be geodesic local coordinates on $M$ with center $y$. We will not distinguish between $\frac{\p}{\p x^i}$ and its canonical 1-lift to $M\times M$.  We will also set $\frac{\p}{\p r } = \frac{\p}{\p r_1 }$.  For a vector field $X$ on $M\times M$, we let $\psi_{;X}$ denote $\nabla_X\psi$. For coordinate vector fields $\frac{\p}{\p x^i}$, we abbreviate this covariant derivative to $\psi_{;i}$ (and similarly interpret $\psi_{L;i}$, etc.). 
%%%%%%%%%%
%%%%%%%%%% LEMMA
%%%%%%%%%%
\begin{lemma}  For derivatives of $\psi_L$, we have
\begin{equation}\label{psiLi}
\psi_{L}^{-1}\psi_{L;i}(x,y) = -\frac{ik}{2} g( Jr\frac{\p}{\p r},\frac{\p}{\p x^i})+ \psi_{L}^{-1}\delta\psi_{L;i},	
\end{equation}
with 
\begin{equation}\label{delpsi}
\psi_{L}^{-1}\delta\psi_{L;i}= -k\frac{ z^a \bar z^b }{12}R(y)(\frac{\p}{\p x^i},r\frac{\p}{\p r}, \frac{\p}{\p z^a},\frac{\p}{\p \bar z^b})-\frac{kz^a\bar z^b }{20}(\nabla_{r\frac{\p}{\p r}}R)(\frac{\p}{\p x^i},r\frac{\p}{\p r},\frac{\p}{\p z^a},\frac{\p}{\p \bar z^b})  + O(kr^5).	
\end{equation}
For derivatives of $\hat\psi$, we have
\begin{equation}\label{nablapsihat}
\hat\psi^{-1}\hat\psi_{;i} =  \frac{1}{2} F^{E}(y)(r \frac{\p}{\p r},\frac{\p}{\p x^i})+ \frac{1}{2} R(y)(r \frac{\p}{\p r},\frac{\p}{\p x^i})+ \hat\psi^{-1} \delta\hat \psi_{;i},	
\end{equation}
with \begin{equation}\label{delpsihat}\hat\psi^{-1} \delta\hat \psi_{;i}\in O(r^2).\end{equation} 
\end{lemma}
%%%%%%%%%%
%%%%%%%%%% END LEMMA

%%%%%%%%%%% PROOF
%%%%%%%%%%%
\begin{proof}
Let $F^V=\pi_1^*F^S-\pi_2^*F^S$ denote the curvature of $V$ induced by the connection on $S$.
Let $\{e_j\}_{j=1}^{2m}$ the 1-lift of a local tangent frame covariant constant along radial geodesics centered at $y$, with $e_1 = \frac{\p}{\p r}$.  The assumption that $\nabla_{\frac{\p}{\p r}}\psi(x,y) = 0$ allows us to write
\begin{equation}\label{difeq1}F^V( r\frac{\p}{\p r},e_j)(x,y)\psi(x,y) =  \nabla_{r\frac{\p}{\p r}}\nabla_{e_j}\psi  + \nabla_{[e_j,r\frac{\p}{\p r}]}\psi
=   \nabla_{ \frac{\p}{\p r}}r\nabla_{e_j}\psi  + \nabla_{\Phi(e_j)}\psi,\end{equation}
 Note because $r\frac{\p}{\p r}$ and $e_j$ are $1$-lifts, we have $F^V( r\frac{\p}{\p r},e_j)(x,y) = F^S(r\frac{\p}{\p r},e_j)(x)$ if we use the same symbol to denote a vector and its lift. Write $\nabla_X\psi = \psi\Gamma^V(X).$   With this notation, we rewrite (\ref{difeq1}) as 
\begin{equation}\psi^{-1}(x,y)F^S( r\frac{\p}{\p r},e_j)(x)\psi(x,y)   
=  \frac{\p}{\p r}\bigl(r\Gamma^V (e_j)\bigr)  + \Gamma^V(\Phi(e_j)).\end{equation}
This translates the equation into an ordinary differential equation on a trivial bundle on $M\times \{y\}$, and we can solve recursively by integrating. 
\begin{equation}\label{recurgamma}\Gamma^V (e_j) = \frac{1}{r}\int_0^r\Bigl(\psi^{-1}F^S(se_1,e_j)\psi - \Gamma^V( \Phi(e_j))\Bigr)ds.\end{equation}

Using the recursion relation (\ref{recurgamma}), we expand (\ref{recurgamma}) as: 
\begin{equation}\label{preokr5}
	\Gamma^V (e_j) =  \frac{1}{r}\int_0^r \psi^{-1}F^S(se_1,e_j)\psi ds
	 - \frac{1}{r}\int_0^r \langle\Phi(e_j),e_i\rangle\frac{1}{s}\int_0^s \Bigl(\psi^{-1}F^S( s_2 e_1,e_i)\psi 
	  -\Gamma^V(   \Phi(e_i))\Bigr)ds_2ds.\end{equation}	

Note that the integral $\frac{1}{r}\int_0^r \langle\Phi(e_j),e_i\rangle\frac{1}{s}\int_0^s \Gamma^V(   \Phi(e_i))ds_2ds$ is $O(kr^5)$, giving  for $X\in \ker(d\pi_2)$, 
\begin{equation}\label{or4}\begin{split}\Gamma^V(X) &=  \frac{1}{2} F^S(y)(r \frac{\p}{\p r},X)  + \frac{r}{3} (\nabla_{\frac{\p}{\p r}} F^S)(y)(r \frac{\p}{\p r},X)\\
& \quad+ \frac{r^2}{4} (\nabla_{ \frac{\p}{\p r}}^2 F^S)(y)(r \frac{\p}{\p r},X) 
-  F^S\Bigl(r\frac{\p}{\p r},\frac{r^2}{24}R(y)( \frac{\p}{\p r},X)\frac{\p}{\p r} \Bigr)+ O(kr^4).
\end{split}\end{equation}

If we now let $\{e_j\}_{j=1}^{2m}$ be the 2-lift of a local tangent frame covariant constant along radial geodesics centered at $x$, with $e_1 = \frac{\p}{\p r_2}$, we may repeat the preceding analysis exchanging $x$ and $y$ and $\frac{\p}{\p r}$ and $\frac{\p}{\p r_2}$ to get for $Y\in \ker(d\pi_1)$, 
\begin{equation}\label{or4Y}\begin{split}\psi_{;Y}(x,y)\psi^{-1}(x,y) &= -\frac{1}{2} F^S(x)(r \frac{\p}{\p r_2},Y)  - \frac{r}{3} (\nabla_{\frac{\p}{\p r_2}} F^S)(x)(r \frac{\p}{\p r_2},Y)\\
& \quad- \frac{r^2}{4} (\nabla_{ \frac{\p}{\p r_2}}^2 F^S)(x)(r \frac{\p}{\p r_2},Y) 
+  F^S\Bigl(r\frac{\p}{\p r_2},\frac{r^2}{24}R(x)( \frac{\p}{\p r_2},Y)\frac{\p}{\p r_2} \Bigr)+ O(kr^4).\end{split}\end{equation}

The preceding discussion did not employ any properties of $S$ beyond the structure of a bundle with connection. Hence we may replace $S$ with any of its factors in order to compute the derivatives of $\hat\psi$ or $\psi_L$.  In the case of $\hat\psi$, the error term improves from $O(kr^4)$ to $O(r^4)$. Letting $V(L^k)$ denote $\Hom(\pi_1^*L^k,\pi_2^*L^k)$, we note that $F^{V(L^k)} = -ik\pi_1^*\omega + ik\pi_2^*\omega$ is covariant constant. Hence  for $X\in \ker(d\pi_2)$, 
\begin{equation}\label{okr5}\Gamma^{V(L^k)} (X) = \frac{-ik}{2} g(Jr \frac{\p}{\p r},X)
-\frac{ik }{24}R(y)(X,r\frac{\p}{\p r},r\frac{\p}{\p r},Jr\frac{\p}{\p r}) -\frac{ik }{40}(\nabla_{r\frac{\p}{\p r}}R)(X,r\frac{\p}{\p r},r\frac{\p}{\p r},Jr\frac{\p}{\p r})  + O(kr^5).\end{equation} 
 \end{proof}

%%%%%%%%%%%%
%%%%%%%%%%%% END PROOF
   
It is useful to record the following special case of (\ref{delpsi}) and (\ref{or4}):
\begin{align}
\notag
\psi_L^{-1}\delta\psi_{L;Jr\frac{\p}{\p r}}+\hat\psi^{-1}\hat\psi_{;Jr\frac{\p}{\p r}}
  =  - iz^a\bar z^b\Bigl[  &F^E(y)( \frac{\p}{\p z^a}, \frac{\p}{\p \bar z^b})+ R(y)( \frac{\p}{\p z^a}, \frac{\p}{\p \bar z^b})+ \frac{2}{3} (\nabla_{r\frac{\p}{\p r}} F^E)(\frac{\p}{\p z^a}, \frac{\p}{\p \bar z^b})\\
\label{spokr5}
&  + \frac{2}{3} (\nabla_{r\frac{\p}{\p r}} R)(\frac{\p}{\p z^a}, \frac{\p}{\p \bar z^b})+\frac{ k z^c\bar z^e }{6}R(y)(\frac{\p}{\p z^c},\frac{\p}{\p \bar z^e}, \frac{\p}{\p z^a}, \frac{\p}{\p \bar z^b})\\
\notag
&  +\frac{ kz^c\bar z^e }{10}(\nabla_{r\frac{\p}{\p r}}R)(\frac{\p}{\p z^c},\frac{\p}{\p \bar z^e}, \frac{\p}{\p z^a}, \frac{\p}{\p \bar z^b})\Bigr]  + O(kr^6+r^4).\end{align}

\begin{proposition}\label{nablapsinablapsi} 
\[\nabla^*\psi^{-1}\nabla\psi =   \frac{ik }{40}(\nabla_{e_j}R)(e_j,r\frac{\p}{\p r},r\frac{\p}{\p r},Jr\frac{\p}{\p r}) 
  - \frac{1}{3}  d_A^* F^E(r \frac{\p}{\p r})  - \frac13 d^*_{\nabla^{LC}}R(r\frac{\p}{\p r}) + O(r^2+kr^4).\]
\end{proposition}
\begin{proof} 
Let $\{e_j\}_j$ be an orthonormal $\nabla_{\frac{\p}{\p r}}$ constant tangent frame, with $\nabla e_j = 0$ at $r=0$. In this frame we have
\begin{align*}
	-\nabla^*\psi^{-1}\nabla\psi &= (\nabla_{e_j}(\psi^{-1}\nabla_{e_j}\psi) - \psi^{-1}\nabla_{\nabla_{e_j}e_j}\psi)\\
&=   \frac{-ik}{2} g(J\nabla_{e_j}r \frac{\p}{\p r},e_j)
-\frac{ik }{24} R(y)(e_j,r\frac{\p}{\p r},e_j,Jr\frac{\p}{\p r})
-\frac{ik }{24} R(y)(e_j,r\frac{\p}{\p r},r\frac{\p}{\p r},Je_j)\\
&\quad
  - \frac{ik }{40}(\nabla_{e_j}R)(e_j,r\frac{\p}{\p r},r\frac{\p}{\p r},Jr\frac{\p}{\p r}) 
  - \frac{ikr}{40}(\nabla_{\frac{\p}{\p r}}R)(e_j,r\frac{\p}{\p r},e_j,Jr\frac{\p}{\p r}) 
  - \frac{ikr}{40}(\nabla_{\frac{\p}{\p r}}R)(e_j,r\frac{\p}{\p r},r\frac{\p}{\p r},Je_j)  
\\
&\quad+ \frac{1}{2} F^E(y)(\nabla_{e_j}r \frac{\p}{\p r},e_j) 
 + \frac{1}{3}  (\nabla_{ e_j} F^E)(y)(r \frac{\p}{\p r},e_j) 
+ \frac{1}{2} R(y)(\nabla_{e_j}r \frac{\p}{\p r},e_j) 
 + \frac{1}{3}  (\nabla_{ e_j} R)(y)(r \frac{\p}{\p r},e_j)  
 + O(r^2+kr^4)\\
&=   \frac{-ik}{2} g(J\Phi(e_j),e_j) - \frac{ik }{40}(\nabla_{e_j}R)(e_j,r\frac{\p}{\p r},r\frac{\p}{\p r},Jr\frac{\p}{\p r}) 
    + \frac{1}{3}   (\nabla_{ e_j} F^E)(y)(r \frac{\p}{\p r},e_j)  + \frac{1}{3}   (\nabla_{ e_j} R)(y)(r \frac{\p}{\p r},e_j) + O(r^2+kr^4)\\
&=    - \frac{ik }{40}(\nabla_{e_j}R)(e_j,r\frac{\p}{\p r},r\frac{\p}{\p r},Jr\frac{\p}{\p r}) 
 + \frac{1}{3}  d_A^* F^E(r \frac{\p}{\p r})
+ \frac{1}{3}  d_{\nabla^{LC}}^* R(r \frac{\p}{\p r})  + O(r^2+kr^4).
\end{align*}
\end{proof}

%%%%%%%%%%% SUB SECTION
\subsection{The Inductive Construction}
We now construct the Schwartz kernel, $q_t$ of our approximation $Q_t$ to $e^{-tD_{A(k)}^2}$. We construct $q_t$ explicitly only in a small neighborhood of the diagonal in $M\times M$. Because $q_t$ is rapidly decreasing away from the diagonal, we will suppress in our discussion the cutoff functions which are needed to extend $q_t$ to the complement of a neighborhood of the diagonal. Fix geodesic normal coordinates, $x$, centered at $y$, and defined in some neighborhood of $y$, $B_y$. We construct the approximation inductively.  Let $r$ denote the geodesic distance from $y$ to $x$, and let
\begin{align*}
	U &:= e^{-\frac{kr^2}{4\tanh(tk)}}\\
	\Prefix&:=(\frac{k}{4\pi\sinh(tk)})^{m}U\sum_{q}e^{kt(m-4q)}P_{2q}
\end{align*}
 we write 
\[q_t(x,y) =  \Prefix  \psi(x,y)\sum_{l=0}^N u_l(x,y)\]
with the $u_l$'s to be determined and $N$ large.

The $u_l(x,y)$'s are sections of $\End(S_y)\simeq \End(\Lambda^{0,\even}T^*_yM\otimes E_y)$ and are constructed inductively so that $Q_0 = I$ and so that 
$\epsilon_t:= (\frac{\p}{\p t}+ D_A^2)q_t$ is sufficiently small.  Write
\begin{equation}
D_A^2 = \nabla^*\nabla + \mathcal{F} + 2  e(F^{0,2})  + 2 e^*(F_A^{0,2}),	
\end{equation}
where, in a local frame $\{Z_j\}_{j=1}^m$ for $T^{1,0}M$ and coframe $\{w^j\}_{j=1}^m$, 
\[\mathcal{F} := -  [e(w^{\bar j}),e^*(w^{\bar l})] F_{\bar j l} =   kg_{\bar j l}[e(w^{\bar j}),e^*(w^{\bar l})]   -  [e(w^{\bar j}),e^*(w^{\bar l})] \hat F_{\bar j l}.\]

On $(0,q)$ forms this becomes 
\[\mathcal{F} =  k(2q-m) +  \mathcal{\hat F},\]
where 
\[\mathcal{\hat F}:=-  [e(w^{\bar j}),e^*(w^{\bar l})] \hat F_{\bar j l}.\]
 
We now compute 
\begin{align*}
\epsilon_t &=(\frac{\p}{\p t}+D_A^2)q_t\\
	 &= \Bigl(\frac{\p}{\p t}+ \nabla^*\nabla + \mathcal{F} + 2  e(F_A^{0,2})  + 2 e^*(F_A^{0,2})  \Bigr) \Bigl( \Prefix  \psi(x,y)\sum_{l=0}^N u_l(x,y)\Bigr).
\end{align*}

It is convenient to conjugate by $\Prefix\psi$.
Note that because $e(F^{0,2})$ raises degree,  $e(F^{0,2})\Prefix  \psi=\Prefix  \psi\bigl( e^{4kt}\psi^{-1}e(F^{0,2})\psi\bigr)$, and $e^*(F^{0,2})\Prefix\psi=\Prefix \psi\bigl( e^{-4kt}\psi^{-1}e^*(F^{0,2})\psi\bigr)$. Conjugating and recalling that in local coordinates  
\[\nabla^*\nabla = -g^{ij}\nabla_i\nabla_j + g^{ij}\Gamma_{ij}^l\nabla_l,\] 
we recast $\epsilon_t$ as
\begin{align*}
\epsilon_t = \psi(x,y) \Prefix  \Biggl(&\frac{\p}{\p t}  +\nabla^*\nabla
+\frac{kr}{ \tanh(tk)}\nabla_{\frac{\p}{\p r}}-\frac{k^2r^2 }{4}-  \frac{k(4m+\Delta (r^2))}{4\tanh(tk)} \\ &
-g^{ij}\Bigl( 2\psi^{-1}\psi_{;i}\nabla_j+ (\psi^{-1}\psi_{;j})_{_;i} +\psi^{-1}\psi_{;i}\psi^{-1}\psi_{;j}\Bigr) + g^{ij}\Gamma_{ij}^l\psi^{-1}\psi_{;l}  \\
 &   +\psi^{-1}\mathcal{\hat F}\psi+ 2e^{4kt}\psi^{-1}e(F_A^{0,2})\psi + 2e^{-4kt}\psi^{-1}e^*(F_A^{0,2})\psi\Biggr) \sum_{l=0}^N  u_l(x,y).
\end{align*}

We use the triviality of the bundle $M\times \End(S_y)$ to replace covariant derivatives by partial derivatives. 
Expanding $\psi^{-1}\psi_{;l}$ as $\bigl(-\frac{ik}{2} g(r J\frac{\p}{\p r},\frac{\p}{\p x^l}) + \psi_L^{-1}\delta\psi_{L;l} + \hat \psi^{-1}\hat \psi_{;l}\bigr)$ as per Equation (\ref{psiLi}) then gives %Ok.(BC)
\begin{align*}
\epsilon_t =\psi(x,y) \Prefix  &\Biggl(\frac{\p}{\p t}  +\Delta
+  ikr J\frac{\p}{\p r}  +\frac{kr}{ \tanh(tk)} \frac{\p}{\p r} 
  -  \frac{k(4m+\Delta (r^2))}{4\tanh(tk)}  \\
&- 2g^{ij}( \psi_L^{-1}\delta\psi_{L;i} + \hat \psi^{-1}\hat \psi_{;i})\frac{\p}{\p x^j}+ \nabla^*\psi^{-1}\nabla\psi\\
&-g^{ij}(\psi_L^{-1}\delta\psi_{L;i} + \hat \psi^{-1}\hat \psi_{;i})(\psi_L^{-1}\delta\psi_{L;j} + \hat \psi^{-1}\hat \psi_{;j})\\
&+ik\bigl(\psi_L^{-1}\delta\psi_{L;r J\frac{\p}{\p r}} + \hat \psi^{-1}\hat \psi_{;r J\frac{\p}{\p r}}\bigr)\\
&+\psi^{-1}\mathcal{\hat F}\psi + 2e^{4kt}\psi^{-1}e(F_A^{0,2})\psi + 2e^{-4kt}\psi^{-1}e^*(F_A^{0,2})\psi\Biggr) \sum_{l=0}^N  u_l(x,y).
\end{align*}

We now make several strongly coordinate dependent choices in our analysis of $\epsilon_t$.  
Set $\Delta_E = -\sum_{j=1}^{2m}\frac{\p^2}{(\p x^j)^2}$. With this notation, we define 
\[L:= \p_t  + \frac{kr}{\tanh(tk)} \frac{\p}{\p r} + ik  r J_0\frac{\p}{\p r} +\Delta_E.\]
This operator will dominate our analysis of $\epsilon_t.$
Define $H$ by 
\[\epsilon_t =\psi  \Prefix    (L + H ) \sum_{l=0}^Nu_l   .\]
We begin the inductive construction by setting $u_0 = I$. Observe that $L$ annihilates $u_0$. In the next section, we define an operator $L^{-1}$ that approximately inverts $L$. We then define $u_l$, $l>0$, inductively by setting
\begin{equation}\label{induct}u_{l+1} = -L^{-1}Hu_l.\end{equation}
We now introduce a filtration on operators which  greatly simplifies our later analysis of the magnitude of the Schwartz kernels produced by this algorithm.  
Define the filtration $W_y^{\cdot}$ on partial differential operators defined in a neighborhood $B_y$ of $y$ as follows. We say that an operator $Z\in W_y^b$ if in geodesic normal coordinates centered at $y$, $Z$ can be expressed as a finite sum
\begin{equation}
Z=\sum_{2p+|I|-|J|\leq b}k^p(x-y)^J\sum_{d\geq 0}e^{4dtk}a_{I,J,p,d}(x,y,tk)\frac{\p}{\p x^I},	
\end{equation}
with $a_{I,J,p,d} = \sum_{j\geq 0}P_{2j+2d}a_{I,J,p,d}$ when $d>0$, and $a_{I,J,p,d}(x,y,tk)$ smooth, and bounded for $t<1$. 
In particular, differentiation by a coordinate vector field has weight $+1$, multiplication by $(x^j-y^j)$ has weight $-1$, multiplication by $k$ has weight $2$, multiplication by $t$ has weight $-2$, multiplication by $e^{4kt}\psi^{-1}e(F_A^{0,2})\psi$ has weight $0$, and $W^i\circ W^j\subset W^{i+j}$.  
Set 
\[H_h:= H -\bigl(2  e^{4kt} \psi^{-1}e(F^{0,2})\psi  + 2  e^{-4kt} \psi^{-1}e^*(F_A^{0,2}) \psi\bigr) . \]
\begin{proposition}\label{hweight}
\[H\in W^0_y.\]
\end{proposition} 
\begin{proof}This claim is an immediate  consequence of Lemma \ref{JJ0} and Equations  (\ref{gij}), (\ref{gammaij}),  (\ref{dr2}), (\ref{delpsi}), (\ref{nablapsihat}), and (\ref{delpsihat}).
\end{proof}
 
\subsection{\texorpdfstring{$L^{-1}$}%Actual title of section vs
{L inverse}}%                        %Title to be displayed in pdf's table of content
In this subsection, we construct an approximate inverse to $L$. 
First compute, for $J,K$ multi-indices, 
\begin{align}\label{Lfirst}
	 L\bigl( a_{JK}(tk)z^J\bar z^K\bigr)
	&=\bigl(\p_t+\Delta_E+(|J|+|K|)\frac{k}{\tanh(tk)}+k(|K|-|J|)\bigr)a_{JK}(tk)z^J\bar z^K\\
	&=\mu_{JK}^{-1}(tk)\bigl(\p_t+\Delta_E)\Bigl(\mu_{JK}(tk) a_{JK}(tk)z^J\bar z^K\Bigr),\notag
\end{align}
where
\begin{align*}
	\mu_{JK}(tk)%&=\exp\Bigl(\int (|J|+|K|)\frac{k}{\tanh(tk)}+k(|K|-|J|) dt\Bigr)\\
	&=\sinh(tk)^{(|J|+|K|)}e^{(|K|-|J|)tk}.
%	&=\bigl(\frac{e^{tk}-e^{-tk}}2)^{|J|+|K|}e^{(|K|-|J|)tk}\\
%	&\simeq\frac{e^{2|K|kt}}2.
\end{align*}
On sections which are polynomial in $z$ and $\bar z$, the  inverse operator is
\begin{align*}
L^{-1}\Bigl(z^J\bar{z}^Ka_{JK}(tk)\Bigr)
&=\int_0^t\int_{\IR^{2m}}\frac{e^{-\frac{|z-y|^2}{4(t-s)}}}{\bigl(4\pi(t-s)\bigr)^m}\frac{\mu_{JK}(sk)}{\mu_{JK}(tk)} y^{J}\bar{y}^K a_{JK}(sk) dyds\\
&=\frac1k\int_0^{tk}\int_{\IR^{2m}}e^{-\pi|y|^2}\frac{\mu_{JK}(s)}{\mu_{JK}(tk)}\Bigl(\sqrt{\frac{4\pi}{k}(tk-s)}y+z\Bigr)^J \Bigl(\sqrt{ \frac{4\pi}{k} (tk-s)}\bar y+ \bar z\Bigr)^Ka_{JK}(s)dyds
%&=  k^{-1}\int_0^{tk} \int_{\IR^n}e^{-\pi|y|^2}\frac{\sinh(s)^{(|J|+|K|)}}{\sinh(tk)^{(|J|+|K|)}}e^{(s-tk)(|K|-|J|)}(\sqrt{\frac{4\pi}{k}(tk-s)}y+z)^J (\sqrt{ \frac{4\pi}{k} (tk-s)}\bar y+ \bar z)^Ka_{JK}(s)dyds.
\end{align*}

Observe that $L^{-1}$ lowers weight by $2$ and increases by $1$ the order of vanishing of $a_{JK}(tk)$  at $0$.   
\begin{proposition} %%%%%%%%%%% PROPOSITION
 $u_l(x,y)\in W_y^{-2l}.$ 
\end{proposition}
\begin{proof} %%%%%%%%%%%% PROOF
This proposition follows immediately from Equation (\ref{induct}) and Proposition \ref{hweight}. 
\end{proof} %%%%%%%%%%%% END PROOF

We extend $L^{-1}$ from polynomials to smooth functions $A(z,\bar z,tk)$ by setting
\[L^{-1}A := L^{-1}p_{2N},\]
where $p_{2N}$ is the degree $2N$ Taylor polynomial for $A$.  Then 
\begin{equation}LL^{-1}-I\in W_y^{-2N}.\end{equation} 
 
\begin{proposition}\label{propasy}  %%%%%%%%%%% PROPOSITION 
For some constants $C_l,c_l, B_l$ depending on the geometry of $M$, 
\[\int_{M\times M}|P_{2q}u_l(x,y)P_{2b}|^2\frac{k^{2m}e^{-\frac{kr^2}{2\tanh(tk)}}e^{2kt(m-4q)}}{\bigl(4\pi\sinh(tk)\bigr)^{2m} }dxdy
\leq
\begin{cases}C_{l }k^{m-2l} ,&   \text{ for } l\geq b+q,\,\, t\geq k^{-1}\\
C_{l }k^{m-2l}  e^{-8kt},&   \text{ for } l<b+q, \,\,t\geq k^{-1}\\
	c_lt^{2l-m},   &\text{ for } l\geq |q-b|, \,\, t< k^{-1},
\end{cases}\]
and
\[ P_{2q}u_l(x,y)P_{2b}=0, \quad \forall l< |q-b|.\]
Moreover
\[\int_M\tr\, P_{2q}u_l(x,x)(\frac{k}{4\pi\sinh(tk)})^{m}e^{kt(m-4q)}dx\leq
\begin{cases}  B_l k^{m-2q } (\frac{e^{kt}}{\sinh(tk)})^{m },&\text{ for } tk>1,\\
	B_l t^{l-m},&\text{ for } tk\leq 1.\end{cases}\] 
\end{proposition}  %%% END PROPOSITION
\begin{proof}   %%%%%%%%%%%% PROOF
Since $u_l$ is a zero order partial differential operator, $u_l(x,y)\in W_y^{-2l}$ implies that 
\[
P_{2q}u_lP_{2b}= \displaystyle\sum_{2p-|J|\leq -2l} \sum_{d\leq q}k^p(x-y)^Je^{4dtk}a_{l,J,p, d}(x,y,tk), \]
with $a_{l,J,p,d}$ smooth and bounded.

When $tk>1$,  there exist constants $c_{l,J,p,d}, C_{l,d}$ so that   we  estimate  
\begin{align*}
\int_{M\times M}|P_{2q}u_l(x,y)P_{2b}|^2&(\frac{k}{4\pi\sinh(tk)})^{2m }e^{-\frac{kr^2}{2\tanh(tk)}}e^{2kt(m-4q)}dxdy\\
&\leq 
\int_{M\times M}\sum_{2p-|J|\leq -2l}\sum_{d\leq q} c_{l,J,p,d}k^{2p}r^{2|J|}(\frac{k}{4\pi\sinh(tk)})^{2m }e^{-\frac{kr^2}{2\tanh(tk)}}e^{2kt(m-4q + 4d)}dxdy\\
&\leq    \sum_{d\leq q}C_{l,d }k^{m-2l}  e^{8kt(- q +  d)}.
\end{align*}
The only way for $u_l$ to acquire a factor of $e^{4qtk}$ is for  $L^{-1}2e^{4tk}e(F_A^{0,2})$ to occur at least $q$ times in its construction. (We note that $(L^{-1}2e^{-4tk}e^*(F_A^{0,2}))^jI$ is \emph{not} exponentially decreasing in general.) This raises degree by $2q$. If $b>0$, then  $L^{-1}2e^{-4tk}e^*(F_A^{0,2})$ must also occur at least $b$ times. This requires $l\geq b+q$.  Hence, when $l<b+q$,  $e^{8kt(-q+d)}\leq e^{-8kt}$.

When $tk<1$, we use the fact that $L^{-1}$ increases the order of vanishing in $tk$ by $1$ to write 
\[a_{l,J,p,d}(x,y,tk) = (tk)^l\hat a_{l,J,p,d}(x,y,tk),\]
with $\hat a_{l,J,p,d}$ smooth.  
Hence, there exist constants $\hat c_{l,J,p},\hat c_l,c_l$ so that  
\begin{align*}
	\int_{M\times M}|P_{2q}u_l(x,y)P_{2b}|^2&(\frac{k}{4\pi\sinh(tk)})^{2m }e^{-\frac{kr^2}{2\tanh(tk)}}e^{2kt(m-4q)}dxdy\\
&\leq 
\int_{M\times M}\sum_{2p-|J|\leq -2l}(tk)^{2l}c_{l,J,p}k^{2p}r^{2|J|}(\frac{k}{4\pi\sinh(tk)})^{2m }e^{-\frac{kr^2}{2\tanh(tk)}}dxdy\\
&\leq    \hat c_{l }t^{2l-m} \frac{ (tk)^{m}}{\sinh(tk)^{m}} \leq    c_{l }t^{2l-m}. 
\end{align*}
The vanishing of $ P_{2q}u_l(x,y)P_{2b}$ for $l< |q-b|$ follows from the observation that the only terms that raise or lower degree in our construction are $2e^{4kt}\psi^{-1}e(F_A^{0,2})\psi$ and $2e^{-4kt}\psi^{-1}e^*(F_A^{0,2})\psi$. To raise or lower degree  by $2q-2b$ requires at least $|q-b|$ applications of $2e^{4kt}\psi^{-1}e(F_A^{0,2})\psi$ or $2e^{-4kt}\psi^{-1}e^*(F_A^{0,2})\psi$ and therefore $|q-b|$ applications of $H$. That many applications of $H$ do not occur in the construction of $u_l$ until $l\geq |q-b|$.

The trace estimate is similar to the Hilbert--Schmidt estimate, with one added complication ---  the only terms in $P_{2q}u_lP_{2q}$  large enough to cancel the $e^{-4tkq}$ in the integrand are those with $L^{-1}2e^{4kt}\psi^{-1}e(F_A^{0,2})\psi$ entering $q$ times. These must also then have $L^{-1}2e^{4kt}\psi^{-1}e^*(F_A^{0,2})\psi$ entering $q$ times to map $(0,2q)$ forms back to $(0,2q)$ forms. Hence, when $tk$ is large, the trace is exponentially decreasing in $tk$ unless $l\geq 2q$. 
\end{proof}%%%%%%%%%%%% END PROOF
Specializing to $t = sk^{-1/2}$, for say $s\in [\frac{1}{4},1],$ (this range merely needs to be $k$-independent) gives
\begin{corollary}\label{hsq}
For some $\alpha >0$, and for all $s\in [\frac{1}{4},1],$
\[\|Q_{sk^{-1/2}}\|_{HS}^2\leq   \alpha^2k^{m}.\]
\end{corollary}

\begin{proposition}\label{hserror}For some $c>0$, depending on the geometry of $M$, we have
\begin{equation}\label{eqn-hserror}
	\|\epsilon_t\|_{HS}^2 \leq c(t^{2N-m}+k^{m-2N}).\end{equation}
\end{proposition}
\begin{proof}By construction, 
\[\epsilon_t = \psi\Prefix \bigl(Hu_{N} + \sum_{l=0}^{N-1}(I-LL^{-1})Hu_l\bigr),\]
and 
$Hu_{N} + \sum_{l=0}^{N-1}(I-LL^{-1})Hu_l\in W_y^{-2N}$.
Moreover, $(I-LL^{-1})Hu_l$ vanishes to order $2N$. 
Hence 
\begin{align*}
\Bigl\| \psi \Prefix \sum_{l=0}^{N-1}(I-LL^{-1})Hu_l\Bigr\|_{HS}^2 
&\leq C \sum_{s\geq 2N}\sum_{2p \leq s-2N}\int_{M\times M}\bigl(\frac{ke^{ kt}}{\sinh(tk)}\bigr)^{2m}
e^{-\frac{k|x-y|^2}{2\tanh(tk)}}k^{2p}|x-y|^{2s}dxdy\\
&\leq C'\sum_{s\geq 2N}k^{m-2N} (\frac{e^{kt}}{\sinh(kt)})^m\tanh^{ s}(tk)\\
&\leq \begin{cases}C'' k^{ m-2N}  & \text{ if }tk\geq 1\\
C'' t^{2N-m}&\text{ if }tk\leq 1.\end{cases}	
\end{align*}
Because $H$ has weight $0$, the estimate for $\|\psi \Prefix Hu_N\|_{HS}^2$ is the same as the first estimate of Proposition \ref{propasy}, giving 
\[\|\psi \Prefix Hu_N\|_{HS}^2\leq \begin{cases}\hat C k^{m-2N} & \text{ if }tk\geq 1,\\ 
\hat C t^{2N-m} & \text{ if }tk\leq 1.\end{cases}\]
Combining these estimates gives the result. 
\end{proof}

%%%%%%%%%%%%%%%%%% PROPOSITION
\begin{proposition}\label{ptoq}For some $a>0$, independent of $k$ large, 
\[\| \Pi - Q_{k^{-1/2}} \|_{HS}\leq ak^{\frac{m-N-1}{2}}.\]
\end{proposition}
%%%%%%%%%%%% PROOF
\begin{proof} We have
\[\| \Pi - Q_{k^{-1/2}} \|_{HS} \leq \| \Pi - e^{-k^{-1/2}D_{A(k)}^2} \|_{HS} + \| e^{-k^{-1/2}D_{A(k)}^2} - Q_{k^{-1/2}} \|_{HS}.\]
By (\ref{hs1}), we can estimate this quantity  by 
\begin{equation}\label{summands}  e^{-\frac{k^{1/2}}{2}}\| e^{-\frac{k^{-1/2}}{2}D_{A(k)}^2}\|_{HS} + \| e^{-k^{-1/2}D_{A(k)}^2} - Q_{k^{-1/2}} \|_{HS} .\end{equation}
 
Using (\ref{duhamel2}) and (\ref{hserror}),
we have, for $s<1$, the  estimate 
\begin{equation}\label{ptoqeq}\| e^{-sk^{-1/2}D_{A(k)}^2} - Q_{sk^{-1/2}}  \|_{HS} \leq  \int_0^{sk^{-1/2}} \|\epsilon_t\|_{HS}dt\leq Ck^{\frac{m-N-1}{2}}.\end{equation}
Using Corollary (\ref{hsq}),  we get the estimate
\begin{align*}
e^{-\frac{k^{1/2}}{2}}\| e^{-\frac{k^{-1/2}}{2}D_{A(k)}^2}\|_{HS} 
&\leq     e^{-\frac{k^{1/2}}{2}}\| Q_{k^{-1/2}}\|_{HS}  +   e^{-\frac{k^{1/2}}{2}}Ck^{\frac{m-N-1}{2}}\\
&\leq 2\alpha e^{-\frac{k^{1/2}}{2}}k^{\frac{m}{2}}
\end{align*}
for the first summand in (\ref{summands}). The desired estimate follows by adding the estimates for the two summands. 
\end{proof}
%%%%%%%%%%%% END PROOF

%%%%%%%%%%%%%%%%%% PROPOSITION
\begin{proposition}\label{ma411} (See \cite[Theorem 4.1.1]{MM}.)
\begin{equation}
\Tr\Pi = \frac{ k^{m}}{2^{ m}\pi^{ m}}\Vol(M)\rk(E) + O(k^{m-1}).
\end{equation}
\end{proposition}
%%%%%%%%%%%% PROOF
\begin{proof}
Because $\Pi$ is a projection, $\Tr\Pi = \|\Pi\|_{HS}^2$. By the preceding proposition, it suffices to compute the Hilbert--Schmidt norm of $Q_{k^{-1/2}}.$ Because we are computing only the leading term, we may ignore all $u_j$ for $j>0$. In particular,
\begin{align*}
\|Q_{k^{-1/2}}\|^2_{HS} &= \int_{M\times M}(\frac{ke^{ kt}}{4\pi \sinh(tk)} )^{2m}
e^{-\frac{k|x-y|^2}{2\tanh(tk)}}\rk(E) dydx +O(k^{m-1})\\
& =  \frac{ k^{m}}{2^{ m}\pi^{ m}}\Vol(M)\rk(E) +O(k^{m-1}).	
\end{align*} 
\end{proof}
%%%%%%%%%%%% END PROOF

%%%%%%%%%%%%%%%%%%%%%%%%%
%%%%%%%%%%%%%%%%%%%%%%%%%  SECTION THE ASYMPTOTICS
%%%%%%%%%%%%%%%%%%%%%%%%%
\section{The Asymptotics}
\subsection{Projections}
Let $E\in S^q_{V,p},$ and let $A$ be an $S^q_{V,p}$-compatible connection. 
Write 
\[\Pi = \sum_{a,b=0}^{\floor{\frac{m}{2}}}\Pi_{2a}^{2b},\]
where $\Pi_{2a}^{2b} = P_{2b}\Pi P_{2a}$, and $\floor{\frac{m}{2}}$ denotes the integer part of $\frac{m}{2}$. Because $\Pi$ is self adjoint, 
\[\Pi_{2b}^{2a} = (\Pi_{2a}^{2b})^*.\]
Because $\Pi$ is a projection, we have 
\[  \Pi_{2a}^{2b}  =  \sum_{\mu =0}^{\floor{\frac{m}{2}}}\Pi_{2\mu}^{2b}\Pi_{2a}^{2\mu}.\]
In particular, 
\[\Pi_{2a}^{2a} = \Pi_{0}^{2a}(\Pi_{0}^{2a})^* + \sum_{\mu >0}\Pi_{2\mu}^{2a}(\Pi_{2\mu}^{2a})^*.\]

\begin{proposition}\label{trths}For $a>0$, 
\begin{equation*}
	\Tr \Pi_{2a}^{2a} = \|\Pi_{0}^{2a}\|_{HS}^2 + O(k^{ m-2a-2}).
\end{equation*}
\end{proposition}
\begin{proof}We know that
\[\Tr \Pi_{2a}^{2a} = \|\Pi_{0}^{2a}\|_{HS}^2 +  \sum_{\mu >0}\|\Pi_{2\mu}^{2a}\|_{HS}^2.\]
By Proposition \ref{ptoq} and Proposition \ref{propasy}, 
\begin{equation}\label{projest}\|\Pi_{2\mu}^{2a}\|_{HS} =  \|P_{2a}Q_{k^{-1/2}}P_{2\mu}\|_{HS}  + O(k^{\frac{m-N-1}{2}})   =  O(k^{\frac{m}{2}-a-\mu}), \forall a>0.\end{equation}
\end{proof}

\begin{corollary} 
$E\in S^q_{V,p}$ with $A$ an $S^q_{V,p}$ compatible connection if and only if 
 $ \sum_{\mu\leq \frac{p}{2}}\|\Pi^{2q+2}_{2\mu}\|_{HS}^2 = O(k^{ m-2q-2-p}).$
\end{corollary}
We henceforth focus our attention on the analysis of $\Pi_0^{2q+2}$. Proposition \ref{ptoq} allows us to consider instead $P_{2q}Q_{k^{-1/2}}P_0$ at the cost of introducing errors with $O(k^{\frac{m-N-1}{2}})$ Hilbert--Schmidt norm.  We therefore assume in the following calculations that $t=k^{-1/2}$.

\begin{proposition}\label{naya}For $a>0$,  
\begin{gather}
\|\Pi_0^{2a} \|_{HS}^2 = k^{m-2a}\frac{2^{-2a}}{(2\pi)^m(a!)^2}\|(F^{0,2}_A )^a\|^2_{L^2}+ O(k^{m-2a - 1}).
\end{gather}  
\end{proposition}
\begin{proof}
Since 
$P_0 u_b^*(x,y)P_{2a}u_l(x,y)P_0\in W^{-2b-2l}_y,$ we have 
\begin{multline*}
\left|\int_{M\times M}   (\frac{k}{4\pi\sinh(tk)})^{2m}U^2e^{2kt(m-4a)}\tr P_0  u_b^*(x,y)P_{2a}u_l(x,y)P_0 dydx\right|\\
 \leq \int_{M\times M} (\frac{k}{4\pi\sinh(tk)})^{2m}e^{-\frac{k}{2}r^2}e^{2mkt }\sum_{2p-|J|\leq -2l-2b}C_{p,J}k^pr^{|J|}  dydx \leq Ck^{m-l-b}.\end{multline*}
Now we estimate 
\begin{align*}
  \|\Pi_0^{2a}\|_{HS}^2 &=    \|P_{2a}Q_{k^{-1/2}}P_{0}\|_{HS}^2 +O(k^{\frac{m-N-1}{2}})\\
&= \int_{M\times M}   (\frac{k}{4\pi\sinh(tk)})^{2m}U^2e^{2kt(m-4a)}\tr P_0\sum_{b=a}^N u_b^*(x,y)P_{2a}\sum_{l=a}^N u_l(x,y)P_0 dydx+O(k^{\frac{m-N-1}{2}})\\
&= \int_{M\times M}   (\frac{k}{4\pi\sinh(tk)})^{2m}U^2e^{2kt(m-4a)}\tr P_0 u_a^*(x,y)P_{2a}u_a(x,y)P_0 dydx+ O(k^{m-2a - 1}), 	
\end{align*}
as long as $N\geq2a$. (We always, of course, choose $N$ sufficiently large.)
Our computation showing that 
\[\left|\int_{M\times M}   (\frac{k}{4\pi\sinh(tk)})^{2m}U^2e^{2kt(m-2a)}\tr P_0  u_b^*(x,y)P_{2a}u_l(x,y)P_0 dydx\right|  \leq Ck^{m-l-b} \] extends immediately to show that any term in $P_0 u_a^*(x,y)P_{2a}u_a(x,y)P_0$ of weight less than $-2a$ contributes at most $O(k^{m-2a - 1})$ to the integral. 

Degree considerations show that 
\begin{equation}\label{uaexp}P_{2a}u_aP_0 = \Bigl(-L^{-1}2e^{4kt}\psi^{-1}(x,y)e(F_A^{0,2}(x))\psi(x,y)\Bigr)^aI.\end{equation}
Observe that $\psi^{-1}(x,y)e(F_A^{0,2}(x))\psi(x,y) = e(F_A^{0,2}(y)) + O(x-y),$ and the $O(x-y)$ term is weight $-1$. Hence we may replace $P_{2a}u_aP_0 $ by $ (-L^{-1}2e^{4kt} e(F_A^{0,2}(y)))^aI $ in the computation of the Hilbert--Schmidt norm, introducing at most an $O(k^{m-2a - 1})$ error. Referring to Lemma \ref{lpowera} below for a computation of $ (-L^{-1}2e^{4kt} e(F_A^{0,2}(y)))^aI $, we see that 
\begin{align*}
  \|\Pi_0^{2a}\|_{HS}^2 &= \int_{M\times M}(\frac{k}{4\pi\sinh(tk)})^{2m}U^2e^{2kt(m-4a)}\tr P_0 \frac{2^{-2a}e^{8atk}}{ k^{2a}(a!)^2}e^*(F_A^{0,2}(y))^ae(F_A^{0,2}(y))^aP_0 dydx  + O(k^{m-2a - 1}) \\
&= k^{m-2a}\int_{M}(\frac{1}{2\pi })^{m} \tr P_0 \frac{2^{-2a}}{ (a!)^2}e^*(F_A^{0,2}(y))^ae(F_A^{0,2}(y))^aP_0 dy  + O(k^{m-2a - 1}),	
\end{align*}
proving the asserted equality. 
\end{proof}

Recall double-factorial notation: 
\[j!!=\begin{cases}j\cdot (j-2)\cdots 1,&\text{ if $j$ is odd,}\\
		   j\cdot (j-2)\cdots 2,&\text{ if $j$ is even.}
\end{cases}\]
\begin{lemma}\label{lpowera}
Suppose $t<1$, $k>>1$ and $tk>>1$.
For $p\geq 0$, we have\begin{align*}
 \Bigl(-L^{-1}2e^{4kt} e\bigl(F_A^{0,2}(y)\bigr)\Bigr)^aI &= (-1)^a\frac{e^{4atk}}{2^{a}k^aa!}e\bigl(F_A^{0,2}(y)\bigr)^a + O(e^{ (4a-1)tk}),\\
L^{-1} e^{4pkt} z^J\bar z^K &=  
\frac{e^{4ptk}}{2 k(2p+|K|)}z^J\bar z^K  + O\Bigl(r^{|J|+|K|}e^{ (4p-1)tk}+ r^{|J|+|K|}e^{4ptk}\sum_{j=0}^{\min\{|J|,|K|\}} \frac{1}{r^{2j}k^j}\Bigr) ,\\
(-\LI 2e^{4kt})^ae^{4pkt}z^J\bar z^K &=\frac{(-1)^a(2p+|K|)!!}{k^a(2p+|K|+2a)!!}e^{4(a+p)tk}z^J\bar{z}^K+O\Bigl(r^{|J|+|K|}e^{ (4(a+p)-1)tk}+r^{|J|+|K|}e^{4ptk}\sum_{j=0}^{\min\{|J|,|K|\}} \frac{1}{r^{2j}k^j}\Bigr).
\end{align*}
\end{lemma}
\begin{proof}We compute 
\[L^{-1}2e^{4ftk}a(y) = \frac{e^{4ftk}-1}{2fk}a(y).\]
Hence 
\[ (-L^{-1}2e^{4kt} e(F_A^{0,2}(y)))^aI = (-1)^a\frac{e^{4atk}}{2^ak^aa!}e(F_A^{0,2}(y))^a + O(e^{(4a-1)tk}).\]
The proof of the second equality is similarly a direct application of the definition of $L^{-1}$. 
\end{proof}

\begin{corollary}\label{corop1} 
$E\in S^q_{V,1}$ with $A$ an $S^q_{V,1}$ compatible connection if and only if 
\[(F_A^{0,2})^{q+1}=0.\] 
\end{corollary}

\begin{corollary}\label{corop1.5} 
If $E\in S^q_{V,1}$, then $ch_p(E)\in (S_H^{p-q}\cap\bar S_H^{p-q})H^{2p}(M,\IQ),\,\,\forall p<q+3.$
\end{corollary}
\begin{proof}Let $A$ be an $S^q_{V,1}$ compatible connection on $A$. We will treat the case $p=q+2$. The other cases follow from similar, albeit simpler, considerations. By Hodge theory, it suffices to show that 
$\tr\, F_A^p$ is a sum of $(s,p-s)$ forms with $2\leq s\leq p-2.$ Expanding $\tr\, F_A^p=\tr(F_A^{2,0}+F_A^{1,1}+F_A^{0,2})^p$ as the sum of the trace of a word in the letters $F^{2,0}_A,F_A^{1,1}$, and $F_A^{0,2}$, we see that it suffices to show that the letter $F_A^{0,2}$ occurs at most $p-2$ times in any word with nonzero trace (and symmetrically $F_A^{2,0}$ occurs at most $p-2$ times in any word with nonzero trace). Clearly, $\tr(F_A^{0,2})^p=0$  by Corollary \ref{corop1}. By the cyclic invariance of the trace, we also have $\tr(F_A^{0,2})^a(sF_A^{1,1}+tF_A^{2,0}) (F_A^{0,2})^{q+1-a}=\tr(F_A^{0,2})^{q+1}(sF_A^{1,1}+tF_A^{2,0}) = 0$, as desired. 
\end{proof}

\begin{corollary}\label{corop2} 
Let  $E\in S^q_{V,1}$ with $S^q_{V,1}$ compatible connection $A$. Then    
\[\|\Pi_{2\mu}^{2q+2}\|^2 = O(k^{m-2q-3-2\mu}).\]
\end{corollary}
\begin{proof}
Using Proposition \ref{propasy} and our approximation $Q_{k^{-1/2}}$ for $\Pi$, we have
\begin{align*}
\|\Pi_{2\mu}^{2q+2}\| &= \|P_{2q+2}Q_{k^{-1/2}}P_{2\mu}\| +O(k^{\frac{k-N-1}{2}})\\
&\leq \sum_{l\geq 2q+2\mu+2} \|P_{2q+2}\psi\Prefix u_l P_{2\mu}\|+\underbrace{\sum_{l< 2q+2\mu+2} \|P_{2q+2}\psi\Prefix u_l P_{2\mu}\|}_{\text{Prop \ref{propasy} }\implies\text{ exp. decay}}+O(k^{\frac{k-N-1}{2}})\\
&=\|P_{2q+2}\psi\Prefix u_{2q+2\mu+2}P_{2\mu}\|+O(k^{m-2q-2\mu-3}),
\end{align*}
given $N$ sufficiently large.
The leading order term of the highest weight term in $u_{2q+2\mu+2}$ arises from 
\[\Bigl(-L^{-1}2e^{4kt} e(F_A^{0,2}(y))\Bigr)^{q+1}\Bigl(-L^{-1}2e^{-4kt}e^*(F_A^{0,2}(y))\Bigr)^{\mu},\] which vanishes by Corollary \ref{corop1}.  The remaining terms have weight less than or equal to $-2q-3-2\mu .$ 
\end{proof}

\begin{proposition}\label{SqV1}   
For $E\in S^q_{V,1}$  with $S^q_{V,1}$ compatible connection $A$, we have
\begin{gather*}
\|\Pi_0^{2q+2}\|_{HS}^2 =\frac{k^{m-2q-3}}{2^{2q+2-m}(4\pi)^m} \left\|\sum_{b=0}^{q}  \frac{2^{b+1}(2q-2b+1)!! (F_A^{0,2})^{b} (\nabla^{0,1} F_{A}^{0,2}) (F_A^{0,2})^{q-b}}{(2q+3)!!(q-b)!}\right\|_{L^2}^2 +O(k^{ m-2q-4}) .
\end{gather*}
\end{proposition}
\begin{proof}
We have seen in the proof of Proposition \ref{naya} that for $t= k^{-1/2}$, we have
\begin{align*}
\|\Pi_0^{2a}\|_{HS}^2& = \int_{M\times M}   (\frac{k}{4\pi\sinh(tk)})^{2m}U^2e^{2kt(m-4a)}\tr P_0\sum_{b=a}^N u_b^*(x,y)P_{2a}\sum_{l=a}^N u_l(x,y)P_0 dydx+O(k^{\frac{m-N-1}{2}})\\
&= \int_{M\times M}   (\frac{k}{4\pi\sinh(tk)})^{2m}U^2e^{2kt(m-4a)}\tr P_0 \Bigl(u_a^*(x,y)+u_{a+1}^*(x,y)\Bigr)P_{2a}\Bigl(u_a(x,y)+u_{a+1}(x,y)\Bigr)P_0 dydx\\
&\quad+O(k^{ m-2a-2}).	
\end{align*}
Setting $a= q+1$, Corollary \ref{corop1} implies the vanishing of $e(F_A^{0,2}(y))^{a}$. Consequently, from the discussion in the proof of Proposition \ref{naya}, we see that 
$u_a(x,y)\in W_y^{-2a-1}$.  This implies  
\begin{equation}\label{2ndcomp}\|\Pi_0^{2a}\|_{HS}^2  = \int_{M\times M}   (\frac{k}{4\pi\sinh(tk)})^{2m}U^2e^{2kt(m-4a)}\tr P_0 u_a^*(x,y)P_{2a}u_a(x,y) P_0 dydx+O(k^{ m-2a-2}).\end{equation}
We Taylor expand in the radial direction:
\begin{align*}
\psi^{-1}(x,y)e(F_A^{0,2}(x))\psi(x,y) = e(F_A^{0,2}(y))  
 + z^\mu e\bigl(\nabla_{\frac{\p}{\p z^\mu}}F_A^{0,2}(y)\bigr)+ \bar z^\mu e\bigl(\nabla_{\frac{\p}{\p \bar z^\mu}}F_A^{0,2}(y)\bigr) + O(|x-y|^2).	
\end{align*}
With the vanishing of $e(F_A^{0,2}(y))^{a}$, the leading order term in $P_{2a}u_aP_0$, as per Equation (\ref{uaexp}), becomes 
\begin{align*}
(-1)^a\sum_{b=0}^{a-1} \Bigl(L^{-1}2e^{4kt} e&\bigl(F_A^{0,2}(y)\bigr) \Bigr)^b\Bigl(L^{-1}2e^{4kt} \Bigl[z^\mu e(\nabla_{\frac{\p}{\p z^\mu}}F_A^{0,2}(y))
+ \bar z^\mu e(\nabla_{\frac{\p}{\p \bar z^\mu}}F_A^{0,2}(y))\Bigr] \Bigr) \Bigl(L^{-1}2e^{4kt} e\bigl(F_A^{0,2}(y)\bigr)\Bigr)^{a-b-1}\displaybreak[0]\\
&= (-1)^a\sum_{b=0}^{a-1}  (L^{-1}2e^{4kt} e(F_A^{0,2}(y)) )^bL^{-1}\frac{2e^{4(a-b)kt}\bar z^\mu e(\nabla_{\frac{\p}{\p \bar z^\mu}}F_A^{0,2}(y)) e(F_A^{0,2}(y))^{a-b-1}}{2^{a-b-1}k^{a-b-1}(a-b-1)!} \\
&\quad\quad\quad\quad+(-1)^a \sum_{b=0}^{a-1}\frac{ e^{4akt}e(F_A^{0,2}(y))^{b}z^\mu e(\nabla_{\frac{\p}{\p   z^\mu}}F_A^{0,2}(y))e(F_A^{0,2}(y))^{a-b-1}}{k^{a }2^aa!} \mod W_y^{-2a-2}\displaybreak[0]\\
&= (-1)^a\sum_{b=0}^{a-1}  \frac{ 2^{b+1 }(2a-2b-1)!!e^{4akt}e(F_A^{0,2}(y))^{b} \bar z^\mu e(\nabla_{\frac{\p}{\p \bar z^\mu}}F_A^{0,2}(y)) e(F_A^{0,2}(y))^{a-b-1}}{k^{a}(2a+1)!!2^{a }(a-b-1)!}\\
&\quad\quad\quad\quad+(-1)^a \sum_{b=0}^{a-1}\frac{ e^{4akt}e(F_A^{0,2}(y))^{b}z^\mu e(\nabla_{\frac{\p}{\p   z^\mu}}F_A^{0,2}(y))e(F_A^{0,2}(y))^{a-b-1}}{k^{a }2^aa!} \mod W_y^{-2a-2}. 	
\end{align*}
The coefficient of $z^\mu$ is a multiple of the $\frac{\p}{\p z^{\mu}}$ covariant derivative of  $0=e(F_A^{0,2})^{a}$, and therefore vanishes.  (The vanishing of the $\bar z$-linear term in the Taylor expansion of $0=e(F_A^{0,2}(y))^{a}$ can also be used to modify the coefficients of $\bar z^\mu$.) Hence
\begin{align*}
P_{2a}u_aP_0 &= (-1)^a\sum_{b=0}^{a-1}  \frac{ 2^{b+1 }(2a-2b-1)!!e^{4akt}e(F_A^{0,2}(y))^{b} \bar z^\mu e(\nabla_{\frac{\p}{\p \bar z^\mu}}F_A^{0,2}(y)) e(F_A^{0,2}(y))^{a-b-1}}{k^{a}(2a+1)!!2^{a }(a-b-1)!} \mod W_y^{-2a-2}. 
\end{align*}
Inserting this equality into Equation (\ref{2ndcomp}) gives 
\begin{align*}	
\|\Pi_0^{2a}\|_{HS}^2 
&= \int_{M\times M}   (\frac{k }{4\pi})^{2m}U^2 \left|\sum_{b=0}^{a-1}  \frac{ 2^{b+1 }(2a-2b-1)!!e(F_A^{0,2}(y))^{b} \bar z^\mu e(\nabla_{\frac{\p}{\p \bar z^\mu}}F_A^{0,2}(y)) e(F_A^{0,2}(y))^{a-b-1}}{k^{a}(2a+1)!!2^{a }(a-b-1)!}P_0\right|^2_{HS}dydx\\
&\quad+O(k^{ m-2a-2}) \\
&= \frac{k^{m-2a-1}}{2^{2a-m}(4\pi)^m} \int_{M} \sum_\mu \left|\sum_{b=0}^{a-1}  \frac{ 2^{b+1 }(2a-2b-1)!!e(F_A^{0,2}(y))^{b} e(\nabla_{\frac{\p}{\p \bar z^\mu}}F_A^{0,2}(y)) e(F_A^{0,2}(y))^{a-b-1}}{(2a+1)!!(a-b-1)!}P_0\right|^2_{HS}dy\\
&\quad+O(k^{ m-2a-2}) 
\end{align*}
To complete the proof, we set $a=q+1$.
\end{proof}

\begin{corollary}\label{peq2}$E\in S^q_{V,2}$ with $A$ an $S^q_{V,2}$ compatible connection if and only if 
\begin{align}
	0 &= (F_A^{0,2})^{q+1},\text{ and}\\		
 0 &= \sum_{b=0}^{q}  \frac{ 2^{b+1 } (2q-2b+1)!!(F_A^{0,2})^{b} (\nabla^{0,1}F_A^{0,2}) (F_A^{0,2})^{q-b }}{(2q+3)!!(q-b )!}.
\end{align}
\end{corollary}

%%%%%%%%%%%%%%%%%%%%%%%%%
%%%%%%%%%%%%%%%%%%%%%%%%%  SECTION METRIC VARIATION
%%%%%%%%%%%%%%%%%%%%%%%%%
\section{Metric Variation}\label{metricvar}
In this section we begin investigating the additional constraints placed on an $S_{V,p}^q$ compatible connection $A$ by $IS_{V,p}^q$ compatibility. These new constraints arise by computing the metric variation of the constraints imposed by $S_{V,p}^q$ compatibility.

Let $g(t)$ be a smooth 1 parameter family of Kahler metrics. Suppose that 
$\dot g(0) = \eta $, for some hermitian 2 tensor $\eta$. Then at $t=0$, 
\begin{equation}\label{varchrist}
\dot\Gamma_{\bar a\bar b}^{\bar c} =  g^{\bar c  e}\eta_{\bar b  e;\bar a},
\end{equation}
Suppose henceforth that 
\[\eta_{a\bar c} = \frac{\p^2H}{\p z^a\p \bar z^c}.\]
This is the form of the metric variation when we vary the line bundle metric $h$ in the polarizing data $(L,h)$. 
At the origin of a K\"ahler normal coordinate system we now have 
\begin{equation}\label{varchristn}
\dot\Gamma_{\bar a\bar b}^{\bar c} =  H_{,\bar a\bar b c}.
\end{equation}
\begin{proposition}\label{isv12}Let $A$ be an $IS^1_{V,2}$ compatible connection and $dim_{\IC}M>3$. Then for every vector field $Z$,
\[ F^{0,2}\wedge i_{Z} F^{0,2}  =0. \]
\end{proposition}
\begin{proof}
Corollary \ref{corop1} specialized to $q=1$ gives \[F^{0,2}\wedge F^{0,2} = 0.\] The derivative of this equality combined with the second equation of Corollary \ref{peq2} gives,
\begin{equation}\label{1deriv}0 =    F_A^{0,2}\wedge \nabla^{0,1}F_A^{0,2},\end{equation}
for $A$ an $S^1_{V,2}$ compatible connection. 
Equation (\ref{1deriv}) couples the connection to the metric via the Levi--Civita action of $\nabla_{\bar Z}$. 
As $A$ is $IS^1_{V,2}$ compatible, the equation holds for all polarizations, and we may differentiate it to obtain
\begin{equation}\label{varq2}
- H_{,\bar a\bar b  c}d\bar z^b\wedge F^{0,2}\wedge F^{0,2}_{\bar c\bar f}d\bar z^f = 0.\end{equation}
The $1$-form $- H_{,\bar a\bar b  c}d\bar z^b$ can take any $(0,1)$ value at a point. Hence, we deduce that for every vector field $Z$, 
\begin{equation}
 F^{0,2}\wedge i_{Z} F^{0,2}  =0. 
\end{equation}
\end{proof}

\begin{proposition}\label{is2com}
If $E\in IS^1_{V,2}$ with compatible connection $A$ and $dim_{\IC}M>3$, then $F_A^{0,2}$ is a $2$-form taking values in a commutative subalgebra of $\ad(E)$. 
\end{proposition}
\begin{proof}
Let $A$ be an $IS_{V,2}^1$ compatible connection. Proposition \ref{isv12} gives  
\begin{equation}d\bar z^c\wedge F^{0,2}_{ \bar a\bar c}\wedge F_A^{0,2}  = 0,\,\,\forall a.\end{equation}
Expanding this equation in coordinates gives 

\begin{equation}\label{p1exp}
 F^{0,2}_{ \bar a\bar c}  F_{\bar b\bar f}^{0,2}
+  F^{0,2}_{ \bar a\bar b}  F_{\bar f\bar c}^{0,2}
+ F^{0,2}_{ \bar a\bar f}  F_{\bar c\bar b}^{0,2} = 0,\,\,\forall a,b,c,f.\end{equation}
When $a=f$ this reduces to 
\begin{equation}
[ F^{0,2}_{ \bar a\bar b} , F_{\bar a\bar c}^{0,2}] = 0,\,\,\forall a,b,c.\end{equation}
A change of coordinates (replacing $a$ by $a+f$) then implies 
\begin{equation}
[ F^{0,2}_{ \bar a\bar b} , F_{\bar c\bar f}^{0,2}]=[ F^{0,2}_{ \bar f\bar b} , F_{\bar a\bar c}^{0,2}] ,\,\,\forall a,b,c,f.\end{equation}
The left hand side of this equality is invariant under the simultaneous exchanges $a\leftrightarrow b$ and $c\leftrightarrow f$, but the right hand side is multiplied by $-1$. Hence 
\begin{equation}
[ F^{0,2}_{ \bar a\bar b} , F_{\bar c\bar f}^{0,2}]=0 ,\,\,\forall a,b,c,f.\end{equation}
\end{proof}

%%%%%%%%%%%
\section{Further asymptotics}
Although we will not do so here, we note that probing the restrictions on $S^q_{V,p}$ compatible connections is sometimes simplified by using the identity 
$\|\Pi_{2a}^{0}\|_{HS}^2 = \|\Pi_{0}^{2a}\|_{HS}^2$ to shift our computations to $\Pi_{2a}^0$ where certain simplifications arise. To see these simplifications, first consider the model computation
\begin{multline*}
L^{-1} z^J \bar z^Ke^{-4ptk} \\
=   \frac{k^{-1}e^{tk(|J|-|K|)}}{\sinh(tk)^{(|J|+|K|)}}\int_0^{tk} \int_{\IR^n}e^{-\pi|y|^2}\sinh(s)^{(|J|+|K|)}e^{s(|K|-|J|-4p)}(\sqrt{\frac{4\pi}{k}(tk-s)}y+z)^J (\sqrt{ \frac{4\pi}{k} (tk-s)}\bar y+ \bar z)^Kdyds.
\end{multline*}
In particular, for $p>0$, $L^{-1} z^J \bar z^Ke^{-4ptk}$ is exponentially decreasing if $|K|\not = 0$. 
Hence, in estimating 
\begin{align*}
  \|\Pi_{2a}^0\|_{HS}^2 &=    \|P_{0}Q_{k^{-1/2}}P_{2a}\|_{HS}^2 +O(k^{\frac{m-N-1}{2}})\\
&= \int_{M\times M}   (\frac{k}{4\pi\sinh(tk)})^{n}U^2e^{2kt(m-4a)}\tr P_{2a}\sum_{b=a}^N u_b^*(x,y)P_{0}\sum_{l=a}^N u_l(x,y)P_{2a} dydx+O(k^{\frac{m-N-1}{2}}),
\end{align*}
we may discard terms in $u_f= (-L^{-1}H)^fI$ arising from exponentially decreasing terms in $H(-L^{-1}H)^{f-1}I$ with $\bar z^K$ factors, $|K|>0.$ Because $L^{-1}$ has terms lowering the degree of a polynomial, we cannot simply remove any term with a $\bar z^b$ factor. Nonetheless, this suggests we analyze the polynomials arising in $H$. 

\subsection{The fine structure of $H$}\label{fineH}
It is convenient to say a monomial differential operator $z^A\bar z^B\frac{\p^{|\alpha|+|\beta|}}{\p z^\alpha\p\bar z^\beta}$ is of type $(|A|-|\alpha|,|B|-|\beta|)$ and to define the \emph{charge} of a monomial differential operator of type $(p,q)$ to be $p-q$. 
We identify the terms in $H$ which raise or lower homogeneity in $\bar z$. First we determine the $O(r^3)$ terms arising in $\Delta r^2  +4m - \frac{2}{3}Ric(r\frac{\p}{\p r},r\frac{\p}{\p r}).$
  Let $\{e_j\}_{j=1}^n$ be an orthonormal tangent frame parallel along radial geodesics emanating from $y$, with $(\nabla e_j)(y) = 0.$ We choose the frame so that $e_j - \frac{\p}{\p x^j} = O(r^2\nabla)$, and therefore 
$e_jr^2 = 2(x^j-y^j) + O(r^3)$.

Then 
\begin{align*}
r\frac{\p}{\p r} (4m+\Delta r^2) &= 2\Delta r^2 -[r\frac{\p}{\p r},e_je_j - \nabla_{e_j}e_j]r^2
\\
&= 2\Delta r^2 +\Bigl(\bigl(e_j+\Phi(e_j)\bigr)e_j+e_j\bigl(e_j+\Phi(e_j)\bigr) +\nabla_{r\frac{\p}{\p r}}\nabla_{e_j}e_j  -\nabla_{e_j}e_j - \Phi(\nabla_{e_j}e_j)\Bigr)r^2
\\
&= \bigl(\Phi(e_j)e_j+e_j\Phi(e_j) +R(r\frac{\p}{\p r},e_j)e_j - \nabla_{ \Phi(e_j)}e_j  - \Phi(\nabla_{e_j}e_j)\bigr)r^2
\\
&=   \frac{4 }{3}Ric(y) ( r\frac{\p}{\p r},r\frac{\p}{\p r})-  \frac{r }{2}(\nabla_{\frac{\p}{\p r} }Ric)(r\frac{\p}{\p r},r\frac{\p}{\p r}) + O(r^4).
\end{align*}
Hence 
\begin{equation}\label{4mDeltar2}
	(4m+\Delta r^2)=    \frac{2}{3}Ric(y)(r \frac{\p}{\p r} ,r\frac{\p}{\p r} ) 
+ \frac{r}{6}(\nabla_{\frac{\p}{\p r} }Ric)(r\frac{\p}{\p r} ,r\frac{\p}{\p r})  + O(r^4).\end{equation}

We now examine $H$. We have
\begin{align*}
H&=H_h+ 2e^{4kt}\psi^{-1}e(F_A^{0,2})\psi + 2e^{-4kt}\psi^{-1}e^*(F_A^{0,2})\psi,
\text{ and}\\
H_h &=   \Delta-\Delta_E+  ikr(J-J_0)\frac{\p}{\p r}
  -  \frac{k(4m+\Delta (r^2))}{4\tanh(tk)}  
- 2g^{ij}( \psi_L^{-1}\delta\psi_{L;i} + \hat \psi^{-1}\hat \psi_{;i})\frac{\p}{\p x^j}+ \nabla^*\psi^{-1}\nabla\psi\\
&\quad-g^{ij}(\psi_L^{-1}\delta\psi_{L;i} + \hat \psi^{-1}\hat \psi_{;i})(\psi_L^{-1}\delta\psi_{L;j} + \hat \psi^{-1}\hat \psi_{;j})
+ik\bigl(\psi_L^{-1}\delta\psi_{L;r J\frac{\p}{\p r}} + \hat \psi^{-1}\hat \psi_{;r J\frac{\p}{\p r}}\bigr)+\psi^{-1}\mathcal{\hat F}\psi.
\end{align*}
A  relatively straightforward but lengthy computation involving numerous identities given so far allows one to get
\begin{align*}
H_h&=-\frac83R(y)(\pp{ z^i},\bar{z}^a\pp{\bar z^a},z^b\pp{z^b},\pp{\bar z^j})\frac{\p^2}{\p{\bar z^i}\p z^j} 
  -\frac43R(y)(\pp{z^i},\bar z^a\pp{\bar z^a},\bar z^b\pp{\bar z^b},\pp{z^j})\frac{\p^2}{\p \bar z^i\p \bar  z^j} 
\\
&\quad  -\frac43R(y)(\pp{\bar z^i},z^a\pp{z^a},z^b\pp{z^b},\pp{\bar z^j})\frac{\p^2}{\p  z^i\p { z^j}}  
+\frac43 Ric(\bar z^a\pp{\bar z^a},\pp{z^j})\pp{\bar z^j}+\frac43 Ric(z^a\pp{z^a},\pp{\bar z^j})\pp{ z^j}\\ 
&\quad-\frac{2k}{3} R(y)(z^b\frac{\p}{\p z^b},\frac{\p}{\p \bar z^\mu},z^l\frac{\p}{\p z^l},\bar z^c\frac{\p}{\p \bar z^c}) \frac{\p}{\p z^\mu} 
+ \frac{2k}{3} R(y)(\bar z^b\frac{\p}{\p \bar z^b},\frac{\p}{\p z^\mu},\bar z^l\frac{\p}{\p \bar z^l},z^c\frac{\p}{\p  z^c} ) \frac{\p}{\p \bar z^\mu}
\\
&\quad +\frac{kz^a\bar z^b}{3}R(y)(\pp{z^j},\bar z^c\pp{\bar z^c},\pp{z^a},\pp{\bar z^b})\pp{\bar  z^j}
+\frac{kz^a\bar z^b}{3}R(y)(\pp{\bar z^j},z^c\pp{ z^c},\pp{z^a},\pp{\bar z^b})\pp{z^j}\\ 
&\quad -2\Bigl(F^E(y)(r\pp r,\pp{\bar z^j})+z^cR(y)(\pp{z^c},\pp{\bar z^j})\Bigr)\pp{z^j} 
-2\Bigl(F^E(y)(r\pp r,\pp{ z^j})+\bar z^cR(y)(\pp{\bar z^c},\pp{ z^j})\Bigr)\pp{\bar z^j}\\  
&\quad
+kz^a\bar z^b\Bigl(F^E(y)( \frac{\p}{\p z^a}, \frac{\p}{\p \bar z^b})+ R(y)( \frac{\p}{\p z^a}, \frac{\p}{\p \bar z^b})\Bigr)
+\frac{k^2z^az^c\bar z^b\bar z^e}6R(y)(\frac{\p}{\p z^c},\frac{\p}{\p \bar z^e}, \frac{\p}{\p z^a}, \frac{\p}{\p \bar z^b}) 
+\psi^{-1}\mathcal{\hat F}\psi 
\\
&\quad -\frac{k\bar{z}^az^b Ric(y)(\pp{\bar z^a} ,\pp{z^b})}{3\tanh(tk)} 
+\delta H_h, 
\text{ with}\\
%%%%
%%% NOW WEIGHT -1
%%%%
\delta H_h& =-\frac{k}{3} (\nabla_{r\frac{\p}{\p r} }R)(z^b\frac{\p}{\p z^b}, \frac{\p}{\p \bar z^\mu}, z^l\frac{\p}{\p z^l},\bar z^c\frac{\p}{\p \bar z^c}) \frac{\p}{\p z^\mu}
+ \frac{k}{3} (\nabla_{r\frac{\p}{\p r} }R)(\bar z^b\frac{\p}{\p \bar z^b}, \frac{\p}{\p z^\mu}, \bar z^l\frac{\p}{\p \bar z^l},z^c\frac{\p}{\p  z^c} ) \frac{\p}{\p \bar z^\mu}\displaybreak[0]\\ 
&\quad +\frac{kz^a\bar z^b}5 (\nabla_{r\pp{r}}R)(\pp{z^j},\bar z^c\pp{\bar z^c},\pp{z^a},\pp{\bar z^b})\pp{\bar  z^j}
      +\frac{kz^a\bar z^b}5 (\nabla_{r\pp{r}}R)(\pp{\bar z^j},z^c\pp{ z^c},\pp{z^a},\pp{\bar z^b})\pp{z^j}\displaybreak[0]\\  
&\quad -\frac{2k z^a\bar{z}^b}{40}(\nabla_{e_j}R)(e_j,r\frac{\p}{\p r},\pp{z^a},\pp{\bar z^b}) 
  - \frac{1}{3}  d_A^* F^E(r \frac{\p}{\p r})  - \frac13 d^*_{\nabla^{LC}}R(r\frac{\p}{\p r}) 
-\frac{k\bar{z}^az^br(\nabla_{\frac{\p}{\p r} }Ric)(\pp{\bar z^a} ,\pp{z^b}) }{12\tanh(tk)}\displaybreak[0]\\ 
&\quad+\frac{2}{3} kz^a\bar z^b \Bigl((\nabla_{r\frac{\p}{\p r} }F^E)(\frac{\p}{\p z^a}, \frac{\p}{\p \bar z^b})+ (\nabla_{r\frac{\p}{\p r} }R)(\frac{\p}{\p z^a}, \frac{\p}{\p \bar z^b})\Bigr)
+\frac{k^2z^az^c\bar z^b\bar z^e}{10}(\nabla_{r\frac{\p}{\p r} }R)(\frac{\p}{\p z^c},\frac{\p}{\p \bar z^e}, \frac{\p}{\p z^a}, \frac{\p}{\p \bar z^b})\displaybreak[0]\\
&\quad+O(r^2\nabla 
+r^3\nabla^2) 
%%%%
%%% NOW WEIGHT -2
%%%%
+ O\Bigl(r^2+kr^4 +kr^5\nabla +\frac{kr^4}{\tanh(tk)}\Bigr). 
\end{align*}
 
This decomposition reflects the fact that $\delta H_h\in W^{-1}_y$.  
Except for the term
$-2F^E(y)(\bar{z}^a\pp{\bar{z}^a},\pp{\bar{z}^j})\pp{z^j}$ which is of type $(-1,1)$ (and thus of charge $-2$), and the term $-2F^E(y)(z^a\pp{z^a},\pp{z^j})\pp{\bar{z}^j}$ which is of type $(1,-1)$ (and thus of charge $2$), the rest of  $H$ (modulo terms of weight $-1$) consists of terms of charge $0$ and has no terms of type $(p,q)$ with $p<0$ or $q<0$.

\begin{proposition}\label{peq3}
If $E\in IS^1_{V,3}$ with $A$ an $IS^1_{V,3}$ compatible connection and $dim_{\IC}M>3$, then   
for all vector fields $Z$, 
\begin{equation}\label{peq31}0 =  F_{A;a}^{0,2} \wedge i_ZF^{0,2}, \end{equation}
and 
\begin{equation}\label{peq32}0 =   F_{A}^{0,2}\wedge   F^{1,1}_{A} \wedge   F_{A}^{0,2} .\end{equation}
\end{proposition}
\begin{proof}The assumption that $A$ is $S^1_{V,3}$ compatible implies that $P_{4}\sum_{l}u_lP_0$ vanishes modulo weight $-7$.
If $r_1$ and $r_2$ are two polynomials of charge $q_1$ and $q_2$ respectively, then 
\[r_1\perp r_2 \text{ in } L_2(e^{-|z|^2/2}),\forall q_1\not = q_2.\]
Hence, the terms of $ P_{4}\sum_{l}u_lP_0$ and $ P_{0}\sum_{l}u_lP_{4}$ of charge $q_i$ each vanish modulo weight $-7$, for $q_i\in \{-2,-1,0,1,2\}$. The assumption that $A$ is $S^1_{V,2}$ compatible implies that 
\begin{equation}\label{S1V2}
	e(F_A^{0,2})^2= e(F_A^{0,2})e(\nabla^{0,1}F_A^{0,2}) =0.
\end{equation}  
The only remaining terms of charge $0$ in $P_{4} u_2P_{0}$, modulo weight $-7$ are,
modulo $O(e^{7kt})$,  
\begin{align*}
L^{-1}2e(F_A^{0,2})e^{4kt}L^{-1}z^a\bar z^be(F_{A;a\bar b+\bar b a}^{0,2})e^{4kt}
&= \frac{e^{8kt}}{30k^2}e(F_A^{0,2}) z^a\bar z^be(F_{A;a\bar b+\bar b a}^{0,2}) 
+  \frac{37e^{8kt}}{900k^3}e(F_A^{0,2})e(F_{A;a\bar a+\bar a a}^{0,2})    \\
&=  \frac{e^{8kt}}{15k^2}e(F_A^{0,2}) z^a\bar z^be(F_{A;\bar b a}^{0,2})
+\frac{e^{8kt}}{30k^2}e(F_A^{0,2}) z^a\bar z^b( F^E_{\bar b a} +R_{\bar b a})e(F_{A}^{0,2}) 
\\
&\quad +  \frac{37e^{8kt}}{450k^3}e(F_A^{0,2})e(F_{A;\bar a a}^{0,2})
+\frac{37e^{8kt}}{900k^3}e(F_A^{0,2})(F^E_{\bar a a}+R_{\bar a a})e(F_{A}^{0,2}) 
  ,\\
L^{-1}z^a\bar z^be(F_{A;a\bar b+\bar b a}^{0,2})e^{4kt}L^{-1}2e(F_A^{0,2})e^{4kt}
&=\frac{e^{8kt}}{20k^2}z^a\bar z^be(F_{A;a\bar b+\bar b a}^{0,2})e(F_A^{0,2})+\frac{e^{8kt}}{50k^3}e(F_{A;a\bar a+\bar a a}^{0,2})e(F_A^{0,2})\\
&=\frac{e^{8kt}}{10k^2}z^a\bar z^be(F_{A;\bar b a}^{0,2})e(F_A^{0,2})
+\frac{e^{8kt}}{20k^2}z^a\bar z^be(R_{\bar b a}F_{A}^{0,2})e(F_A^{0,2}) \\
&\quad- \frac{e^{8kt}}{20k^2}z^a\bar z^b e(F_{A}^{0,2})F^E_{\bar b a}e(F_A^{0,2})
 +\frac{e^{8kt}}{25k^3}e(F_{A;\bar a a}^{0,2})e(F_A^{0,2})\\
&\quad+\frac{ e^{8kt}}{50k^3}e(R_{\bar a a}F_{A}^{0,2})e(F_A^{0,2})-\frac{ e^{8kt}}{50k^3}e(F_{A}^{0,2})F^E_{\bar a a}e(F_A^{0,2}),\\
L^{-1}2z^ae(F_{A;a}^{0,2})e^{4kt}L^{-1}2\bar z^be(F_{A;\bar b}^{0,2})e^{4kt}
 &= \frac{e^{8kt}}{15k^2}z^ae(F_{A;a}^{0,2})\bar z^be(F_{A;\bar b}^{0,2}) + \frac{2e^{8kt}}{75k^3} e(F_{A;a}^{0,2})e(F_{A;\bar a}^{0,2}) ,\\
L^{-1}2\bar z^be(F_{A;\bar b}^{0,2})e^{4kt}L^{-1}2  z^ae(F_{A;a}^{0,2})e^{4kt}&=
\frac{e^{8kt}}{10k^2}\bar z^be(F_{A;\bar b}^{0,2})z^ae(F_{A;a}^{0,2}) + \frac{ e^{8kt}}{25k^3} e(F_{A;\bar a}^{0,2}) e(F_{A;a}^{0,2}).
\end{align*}

The only remaining term of charge $0$ in $P_{4} u_3P_{0}$, modulo weight $-7$ is, modulo $O(e^{7kt})$,
\begin{multline*}
-L^{-1}2  e(F_{A}^{0,2})e^{4kt}L^{-1}(\mathcal{\hat F}(y)+k z^a\bar z^b[F^E_{a\bar b} + R_{a\bar b}])L^{-1}2e(F_{A}^{0,2})e^{4kt}\\
 =  -\frac{e^{8kt}}{32k^3}  e(F_{A}^{0,2}) \mathcal{\hat F}(y) e(F_{A}^{0,2}) 
 - \frac{e^{8kt}}{60k^3}  e(F_{A}^{0,2}) k z^a\bar z^b[F^E_{a\bar b} + R_{a\bar b}] e(F_{A}^{0,2}) 
 -\frac{37e^{8kt} }{1800k^3}  e(F_{A}^{0,2}) [F^E_{a\bar a} + R_{a\bar a}]  e(F_{A}^{0,2}) .
\end{multline*}

Differentiating (\ref{S1V2}), we get that $F^{0,2}_{A;a}F^{0,2}_{A;\bar b}=-F^{0,2}_AF^{0,2}_{A;\bar b a}$ and $F^{0,2}_{A;\bar b}F^{0,2}_{A;a}=-F^{0,2}_{A;\bar b a}F^{0,2}_A$.  These equations allow us to cancel some of the contributions above.  The total leading order contribution from charge zero is then 
\begin{align*}
 \nu_0&:= \frac{e^{8kt}}{16k^3}  e(F_{A}^{0,2}) e(d\bar z^b)e^*(d\bar z^a)(F^E_{\bar ba}+R_{\bar b a}) e(F_{A}^{0,2}) 
+\frac{e^{8kt}}{18k^3}e(F^{0,2}_{A})e(F^{0,2}_{A;\bar aa})
%&\quad 
 +\frac{e^{8kt} }{24k^3}  e(F_{A}^{0,2}) F^E_{\bar aa} e(F_{A}^{0,2})
+\frac{e^{8kt} }{24k^3}  e(F_{A}^{0,2})  R_{\bar aa} e(F_{A}^{0,2}).
\end{align*}

The assumption that $A$ is $S^1_{V,3}$ compatible implies $\nu_0 = 0$. Note that
$[R_{\bar a b},F_A^{0,2}]=2g^{s\bar t}R_{\bar ab t\bar l}d\bar z^l\wedge i_{\pp{\bar z^s}}F^{0,2}$.
When  $A$ is $IS^1_{V,2}$-compatible then the identities in Proposition \ref{isv12} force $F_A^{0,2}[R_{\bar a b},F_A^{0,2}]=0$.

Therefore, under the assumption that $A$ is $IS^1_{V,2}$-compatible,  $\nu_0$  reduces to
\[\check\nu_0:=  -\frac{e^{8kt}}{18k^3}e(F_{A;a}^{0,2}) e(F_{A;\bar a }^{0,2})  
 - \frac{e^{8kt}}{24k^3}e(F_A^{0,2}) F^E_{\bar a a} e(F_{A}^{0,2}) +\frac{e^{8kt}}{16k^3}  e(F_{A}^{0,2}) e(d\bar z^b)e^*(d\bar z^a) F^E_{\bar b a} e(F_{A}^{0,2}).\]

Finally, we want to exploit the full assumption that $A$ is $IS^1_{V,3}$-compatible. Thus we assume  $\check\nu_0$ vanishes for all polarizations. First we rewrite $0=-144k^3e^{-8kt}\check\nu_0$ with metric terms explicit rather than hidden in orthonormal coordinate systems: 
\[0=  8g^{a\bar b}e(F_{A;a}^{0,2}) e(F_{A;\bar b }^{0,2})  
+  6e(F_A^{0,2})g^{a\bar b} F^E_{\bar b a} e(F_{A}^{0,2}) 
-9 e(F_{A}^{0,2}) e(d\bar z^b) g^{a\bar \mu}F^E_{\bar b a} i_{\pp{\bar z^\mu}}e(F_{A}^{0,2}).\]

Writing $F^{0,2}$ in an anti-holomorphic frame, we see that $F^{0,2}_{;a}$ is independent of the metric.  Varying the metric gives 
\begin{align*}
0&=8\dot g^{a\bar b}F_{A;a}^{0,2} F_{A;\bar b }^{0,2}
+ 32 g^{a\bar b}F_{A;a}^{0,2} H_{,\bar b\bar \mu c}d\bar z^\mu\wedge F_{\bar c\bar f} d\bar z^f+
  6F_A^{0,2}\dot g^{a\bar b} F^E_{\bar b a} F_{A}^{0,2}  
 -9 F_{A}^{0,2}\wedge d\bar z^b \dot g^{a\bar \mu}F^E_{\bar b a}\wedge i_{\pp{\bar z^\mu}}F_{A}^{0,2}.
\end{align*}

At a fixed point, we may simultaneously choose $\dot g = 0$ and $\db H_{,\bar b c}$ an arbitrary $(0,1)$ form. For such a choice, 
\[0 = \db H_{,\bar b c}\wedge   F_{A;b}^{0,2} \wedge F^{0,2}_{\bar c\bar f}d\bar z^f. \]
Hence when $m>3$, for all vector fields $Z$, 
\[0 =  F_{A;a}^{0,2} \wedge i_ZF^{0,2}_A. \]
Now choose $\dot g$ arbitrary at a fixed point, and 
the remaining terms in the variation of $\check\nu_0$ give
\[0 =  8 F_{A;a}^{0,2}\wedge F_{A;\bar b }^{0,2}   
 +6 F_A^{0,2}\wedge  F^E_{\bar b a}  F_{A}^{0,2}  -9 F_{A}^{0,2}\wedge   F^E_{\bar \mu a} d\bar z^\mu\wedge  F_{\bar b\bar c}^{0,2}d\bar z^c .\]

Wedging with $dz^a\wedge d\bar z^b$ and summing over $a$ and $b$ yields Equation (\ref{peq32}).\end{proof}

\begin{proposition}\label{incpfA}
If $E\in S_{V,1}^1,$ then 
\begin{align*}
\ch_p(E)&\in (S_H^{p-\lfloor\frac{p}{2}\rfloor}\cap \bar S_H^{p-\lfloor\frac{p}{2}\rfloor})H^{2p}(M,\IQ), \text{ for all $p$}. 
\intertext{If $E\in IS_{V,3}^1,$ then}
\ch_p(E)&\in (S_H^{p-1}\cap \bar S_H^{p-1})H^{2p}(M,\IQ) \text{ for all $p<7$.}
\end{align*}
\end{proposition}
\begin{proof} 
Let $S_H^aC^{p}(M,\IC)$ denote the $p-$forms which can be written as a sum of $(s,p-s)$ forms, $s\geq a$. Let $\bar S_H^aC^{p}(M,\IC)$ denote the conjugate filtration. 
To show that $\ch_p(E)\in (S_H^{p-1}\cap \bar S_H^{p-a})H^{2p}(M,\IQ)$,  it suffices to show that 
$\tr (F_{A}^{2,0} + F_{A}^{1,1}+F_A^{0,2})^p$ is cohomologous to an element of $(S_H^{p-a}\cap \bar S_H^{p-a})C^{p}(M,\IC)$.  Expand this trace as the sum of traces of words in the letters 
$F_{A}^{2,0}$, $F_{A}^{1,1}$, and $F_A^{0,2}$. 
Let $A$ be an $S_{V,1}^1$ compatible connection. Then Corollary \ref{corop1} implies that $F_A^{0,2}\wedge F_A^{0,2}=0.$ Hence, after any cyclic rearrangement of a monomial with nonzero trace, there must be an $F_A^{2,0}$ or $F_A^{1,1}$ factor between any two $F_A^{0,2}$ factors. Consequently, at most $\lfloor\frac{p}{2}\rfloor$ $F_A^{0,2}$ factors may appear in  any monomial of degree $p$ with nonzero trace.  This proves $\ch_p(E)\in S_H^{p-\lfloor\frac{p}{2}\rfloor}H^{2p}(M,\IQ)$ for all $p$. The conjugate inclusion follows similarly, proving the first assertion. 

Now assume that $A$ is an $IS_{V,3}^1$ compatible connection.  
To prove the second assertion, we consider the case $p=6$, as the case $p<6$ follows from similar but simpler arguments.  We have seen that the trace of any nonzero monomial of degree $6$ in the curvature components must have at most $3$ $F_A^{0,2}$ factors, and (for every cyclic rearrangement) there must be an $F_A^{2,0}$ or $F_A^{1,1}$ factor between any two such factors. If there are at least $2$ $F_A^{2,0}$ factors, then the trace of the monomial lies in $S_H^5C^{12}(M,\IC)$. By proposition \ref{peq3}, $F_{A}^{0,2}\wedge F_{A}^{1,1}\wedge F_{A}^{0,2}=0$ for $IS_{V,3}^1$ compatible $A$ . Hence between any two $F_{A}^{0,2}$ factors there must be an $F_A^{2,0}$ factor or an $(F_A^{1,1})^2$ factor. The case of $3$ $F_A^{0,2}$ factors, $1$ or $2$ $F_A^{1,1}$ factors and $2$ or $1$ (respectively) $F_A^{2,0}$ is therefore excluded. Thus the only monomials with nonzero trace and $3$ $F_A^{0,2}$ factors are cyclic rearrangements of $(F_A^{0,2}\wedge F_A^{2,0})^3.$ These terms yield $(6,6)$ forms. Terms with at most $1$ $F_A^{0,2}$ factor lie in  $S_H^5C^{12}(M,\IC)$, as do terms with $1$ $F_A^{2,0}$ factor and $2$ $F_A^{0,2}$ factors. Hence we are left to consider terms with exactly $2$ $F_A^{0,2}$ factors and $4$  $F_A^{1,1}$ factors. The only monomials of this form with nonzero trace are cyclic rearrangements of $(F_A^{1,1})^2\wedge F_A^{0,2}\wedge (F_A^{1,1})^2\wedge F_A^{0,2}.$  We need now to show that the trace of this term is cohomologous to an element of $S_H^5C^{12}(M,\IC)$. 
We can use Proposition \ref{peq3} to set $F_A^{0,2}\wedge \p_AF_A^{0,2} = 0$ and the Bianchi identity to replace $\db_AF_A^{1,1}$ by $-\p_AF_A^{0,2}$. We thus have 
\begin{align*}
\tr (F_A^{1,1})^2\wedge F_A^{0,2}\wedge(F_A^{1,1})^2 \wedge F_A^{0,2} 
&= \tr (F_A^{1,1})^2\wedge F_A^{0,2}\wedge  F_A^{1,1} \wedge\db_A\p_AF_A^{0,2}\\
&= \tr (F_A^{1,1})^2\wedge\db_A[F_A^{0,2}\wedge F_A^{1,1} \wedge \p_AF_A^{0,2}] - \tr (F_A^{1,1})^2 \wedge F_A^{0,2}\wedge \db_AF_A^{1,1} \wedge  \p_AF_A^{0,2} \\
&= \tr \db_A[(F_A^{1,1})^2\wedge F_A^{0,2}\wedge  F_A^{1,1} \wedge \p_AF_A^{0,2}] - \tr \db_AF_A^{1,1}\wedge F_A^{1,1}\wedge F_A^{0,2}\wedge F_A^{1,1} \wedge \p_AF_A^{0,2}  \\
&\quad - \tr F_A^{1,1}\wedge \db_A F_A^{1,1}\wedge F_A^{0,2}\wedge  F_A^{1,1} \wedge \p_AF_A^{0,2} + \tr (F_A^{1,1})^2 \wedge F_A^{0,2}\wedge \p_AF_A^{0,2}\wedge  \p_AF_A^{0,2}  \\
&= \db\tr  (F_A^{1,1})^2\wedge F_A^{0,2}\wedge  F_A^{1,1} \wedge \p_AF_A^{0,2}  \\
&= d\tr  (F_A^{1,1})^2\wedge F_A^{0,2}\wedge  F_A^{1,1} \wedge \p_AF_A^{0,2}  \mod S_H^5C^{12}(M,\IC)  . 	
\end{align*}

 The result follows. 
\end{proof}

Having exploited the terms in $ P_{4}\sum_{l}u_lP_0$ of charge $0$, we now turn to those of charge $-2$. 
\begin{proposition}Let $E\in S^1_{V,3}$ with compatible connection $A$. 
Then for all $a,b$,
\begin{equation}
	F_{A;\bar a}^{0,2}\wedge F_{A;\bar b}^{0,2}  +  F_{A;\bar b}^{0,2}\wedge F_{A;\bar a}^{0,2}=0.
\end{equation}
For $E\in IS^1_{V,3}$ with compatible connection $A$, for all $a,c$, 
\begin{align}
\label{iFdbF}i_{\frac{\p}{\p \bar z^c}}F_{A}^{0,2}\wedge F_{A;\bar a}^{0,2}   =0,\text{ and}\\	
(i_{\frac{\p}{\p \bar z^c}}F_{A}^{0,2})\wedge (i_{\frac{\p}{\p \bar z^e}}F_{A}^{0,2}) =0.
\end{align}
\end{proposition}
\begin{proof}
The terms of charge $-2$ in $P_{4} u_2P_{0}$, modulo weight $-7$ and $O(e^{7kt})$
are
\begin{align*}
L^{-1}2\bar z^ae^{4kt}e(F_{A;\bar a}^{0,2})L^{-1}2\bar z^be^{4kt}e(F_{A;\bar b}^{0,2}) 
&= \frac{e^{8kt}}{18k^2}\bar z^ae(F_{A;\bar a}^{0,2})\bar z^be(F_{A;\bar b}^{0,2}),\\
L^{-1}2e^{4kt}e(F_{A}^{0,2})L^{-1}e^{4kt}\bar z^a\bar z^be(F_{A;\bar a\bar b}^{0,2})
&= \frac{e^{8kt}}{48k^2}e(F_{A}^{0,2})\bar z^a\bar z^be(F_{A;\bar a\bar b}^{0,2}),
\text{ and}\\
L^{-1}e^{4kt}\bar z^a\bar z^be(F_{A;\bar a\bar b}^{0,2})L^{-1} 2e^{4kt}e(F_{A}^{0,2})
&= \frac{e^{8kt}}{24k^2}\bar z^a\bar z^be(F_{A;\bar a\bar b}^{0,2})e(F_{A}^{0,2}).	
\end{align*}
The assumption that $E\in S^1_{V,2}$ implies that 
\[0 = e(F_{A;\bar a}^{0,2}) e(F_{A }^{0,2}).\]
Using these equalities, we write the charge $-2$ contribution (modulo weight $-7$ and $O(e^{7kt})$)  as 
\[ \frac{e^{8kt}\bar z^a\bar z^b}{6k^2}(\frac{1}{3}e(F_{A;\bar a}^{0,2})e(F_{A;\bar b}^{0,2})+\frac{1}{8}e(F_{A}^{0,2})e(F_{A;\bar a\bar b}^{0,2})+ \frac{1}{4}e(F_{A;\bar a\bar b}^{0,2})e(F_{A}^{0,2})) = 
-\frac{e^{8kt}\bar z^a\bar z^b}{144k^2}e(F_{A;\bar a}^{0,2})e(F_{A;\bar b}^{0,2}).\]
This term must vanish for $E\in S^1_{V,3}$. Hence 
\begin{equation}\label{symeq} F_{A;\bar a}^{0,2}\wedge F_{A;\bar b}^{0,2}  +  F_{A;\bar b}^{0,2}\wedge F_{A;\bar a}^{0,2} =0.\end{equation}

Suppose now that $A$ is $IS_{V,3}^1$ compatible. Then we may take the first variation of the preceding equation with $b=a$ to obtain (no $a$-sum)
\[ \sum_c\db H_{,\bar a c}\wedge [(i_{\frac{\p}{\p \bar z^c}}F_{A}^{0,2})\wedge F_{A;\bar a}^{0,2} +  F_{A;\bar a}^{0,2}\wedge i_{\frac{\p}{\p \bar z^c}}F_{A}^{0,2}]  =0.\]
Hence for all $(a,c)$, 
\begin{equation}\label{anothervar} (i_{\frac{\p}{\p \bar z^c}}F_{A}^{0,2})\wedge F_{A;\bar a}^{0,2} +  F_{A;\bar a}^{0,2}\wedge i_{\frac{\p}{\p \bar z^c}}F_{A}^{0,2}  =0.\end{equation}
Since by Proposition \ref{is2com}, $F^{0,2}_A$ is a form taking values in a commutative subalgebra of $\ad(E)$, this equation implies Equation (\ref{iFdbF}).  In turn,
taking the first variation of  equation (\ref{iFdbF}), we find 
\begin{equation}\label{anothervar2} (i_{\frac{\p}{\p \bar z^c}}F_{A}^{0,2})\wedge (i_{\frac{\p}{\p \bar z^e}}F_{A}^{0,2}) =0.\end{equation}
\end{proof}

\begin{corollary}
Let $E\in IS^1_{V,3}$ with compatible connection $A$. Then $F_A^{0,2}$ takes values in a commutative nilpotent subalgebra of $\End(E)$ whose elements all square to zero. 
\end{corollary}

%%% Bibliography

%\bibliographystyle{plain}
%\bibliography{asymphodges-bib.bib}
%\end{document}

 \end{document}